\documentclass[11pt,reqno]{amsart}
\usepackage{amssymb,amsfonts,amsthm,amscd,stmaryrd,dsfont,esint,upgreek,constants,todonotes}
\usepackage{amsmath,amsthm,amssymb,amsfonts}
\usepackage[shortlabels]{enumitem}
\usepackage{mathtools, mathrsfs, xargs}
\usepackage{hyperref}
\hypersetup{colorlinks, citecolor=blue, linkcolor=blue, urlcolor=blue}
\newif\iflabels
\labelsfalse 
\iflabels 
  \usepackage[left=0.5in, right=2.5in, marginparwidth=2.2in]{geometry}
  \usepackage{showlabels}
\else
  \usepackage[left=1.25in, right=1.25in]{geometry}
\fi

\numberwithin{equation}{section}
\theoremstyle{plain}                
\newtheorem{theorem}{Theorem}[section]
\newtheorem{lemma}[theorem]{Lemma}
\newtheorem{proposition}[theorem]{Proposition}

\theoremstyle{definition}           
\newtheorem{definition}[theorem]{Definition}

\theoremstyle{remark}
\newtheorem{remark}[theorem]{Remark}

\providecommand{\alias}{}
\renewcommand{\alias}[1]{\providecommand{#1}{}\renewcommand{#1}}

  \DeclarePairedDelimiter\ab{\langle}{\rangle} 
  \DeclarePairedDelimiter\abs{\lvert}{\rvert}   
  \DeclarePairedDelimiter\norm{\lVert}{\rVert}  

  \DeclarePairedDelimiterX\set[1]\{\}{ #1 }
  \DeclarePairedDelimiterX\sets[2]\{\}{ #1\,:\,#2 }

  \let\bPeexp\exp
  \let\exp\relax
  \DeclarePairedDelimiterXPP\exp[1]{\bPeexp}(){}{#1}

  \makeatletter
    \let\oldabs\abs \def\abs{\@ifstar{\oldabs}{\oldabs*}}
    \let\oldab\ab \def\ab{\@ifstar{\oldab}{\oldab*}}
    \let\oldnorm\norm \def\norm{\@ifstar{\oldnorm}{\oldnorm*}}
    \let\oldexp\exp \def\exp{\@ifstar{\oldexp}{\oldexp*}}
  \makeatother

\DeclareMathOperator*\esssup{esssup}

\makeatletter
  \newcommand{\opnorm}{\@ifstar\@opnorms\@opnorm}
  
  \newcommand{\@opnorm}[2][]{%
    \mathopen{#1|\mkern-1.5mu#1|\mkern-1.5mu#1|}
    #2
    \mathclose{#1|\mkern-1.5mu#1|\mkern-1.5mu#1|}
  }
\makeatother

\alias{\R}{{\mathbb R}}
\alias{\C}{{\mathbb C}}
\alias{\Z}{{\mathbb Z}}
\alias{\N}{{\mathbb N}}

\newcommand{\ind}[1]{ 1_{{#1}}} 

\newcommand{\eforall}{\text{ for all }}

\newcommand{\ewhere}{\text{ where }}

\newcommand{\sA}{\mathcal{A}}
\newcommand{\sP}{\mathcal{P}}
\newcommand{\sS}{\mathcal{S}}

\newcommand{\sF}{\mathcal{F}}
\newcommand{\sG}{\mathcal{G}}

\newcommand{\bP}{\mathbb{P}}

\newcommand{\bE}{\mathbb{E}}
\newcommand{\bQ}{\mathbb{Q}}

\newcommand{\kZ}{\kZ}

\newcommand{\ltwo}{L^2}
\newcommand{\lpee}{L^p}
\newcommand{\lque}{L^q}
\newcommand{\linf}{L^{\infty}}

\newcommand{\sinf}{\sS^{\infty}}

\newcommand{\stwo}{\sS^2}
\newcommand{\spee}{\sS^p}

\newcommand{\BMO}{\text{BMO}}
\newcommand{\bmo}{\text{bmo}}
\newcommand{\bmoh}{\bmo^{1/2}}

\newcommand{\const}[1]{C=C(#1)}

\newcommand{\sam}{\set{a_m}}

\newcommand{\calpha}{C^{0,\alpha}}

\newcommand{\cbeta}{C^{0,\beta}}

\newcommand{\ption}{(\tau_k)_{k=0}^m}
\newcommand{\itauk}{\ind{ [\tau_{k-1}, \tau_k] }}

\newcommand{\tr}{\text{tr}}

\newcommand{\osc}{\text{osc}}

\newcommand{\be}{\begin{equation}}
\newcommand{\ee}{\end{equation}}

\newcommand{\all}{[0,T]\times \R^d \times \R^n \times (\R^d)^n}

\newcommand{\ct}{C_{0}}
\newcommand{\alphat}{\alpha_{0}}

\begin{document}
\title[Quasilinear parabolic systems and quadratic FBSDEs]{On quasilinear parabolic systems and FBSDEs of quadratic growth}
\author{Joe Jackson}
\address{Department of Mathematics, The University of Texas at Austin}
\email{jjackso1@utexas.edu}
\thanks{During the preparation of this work the author has been supported by the National Science Foundation under Grant No. DGE1610403 (2020-2023). Any opinions, findings and conclusions or recommendations expressed in this material are those of the author(s) and do not necessarily reflect the views of the National Science Foundation (NSF).
}

\maketitle


\begin{abstract}
Using probabilistic methods, we establish a-priori estimates for two classes of quasilinear parabolic systems of partial differential equations (PDEs). We treat in particular the case of a nonlinearity which has quadratic growth in the gradient of the unknown. As a result of our estimates, we obtain the existence of classical solutions of the PDE system. From this, we infer the existence of solutions to a corresponding class of forward-backward stochastic differential equations.
\end{abstract}

\section{Introduction}

We present a-priori estimates and well-posedness results for two classes of quasi-linear parabolic systems. The first reads
\begin{align} \label{pde}
    \begin{cases}
    \partial_t u^i +  \tr(a(t,x,u) D^2u^i) + f^i(t,x,u,Du) = 0, \quad (t,x) \in (0,T) \times \R^d, \vspace{.1cm} \\ 
    u^i(T,x) = g^i(x), \quad x \in \R^d,
    \end{cases}
\end{align}
for $i = 1,...,n$. The data consists of functions $a$, $f$, and $g$,
and the unknown is a map $u = u(t,x) = (u^i(t,x))_{i = 1,...,n} : [0,T] \times \R^d \to \R^n$. Precise assumptions will be given below, but we are particularly interested in the case that $a = \frac{1}{2} \sigma \sigma^T$ is non-degenerate and $f = f(t,x,u,p)$ exhibits quadratic growth in the variable $p$. While \eqref{pde} is the main object of the paper, it turns out that roughly the same methods yield estimates and existence results also for the equation
\begin{align} \label{pde2}
    \begin{cases}
    \partial_t u^i +  a(t,x,u,Du) D^2u^i + f^i(t,x,u,Du) = 0, \quad (t,x) \in (0,T) \times \R, \vspace{.1cm} \\
    u^i(T,x) = g^i(x), \quad x \in \R, 
    \end{cases}
\end{align}
for $i = 1,...,n$.
The key difference between \eqref{pde} and \eqref{pde2} is that the gradient of $u$ appears as an argument in the function $a$, which makes the analysis much more difficult. Accordingly, our techniques apply to \eqref{pde2} only in one spatial dimension and under the assumption that the driver $f$ is globally Lipschitz in $(x,u,p)$.

Systems of the type \eqref{pde} are well-studied, and a classical reference is \cite{lady}. For example, Theorem 7.1 of \cite{lady} gives the existence of a classical solution to a system similar to \eqref{pde}, but on a bounded spatial domain and under the assumption that $f$ has strictly subquadratic growth in $p$. More recently, motivated largely by applications to stochastic differential games, Bensoussan and Frehse undertook an intensive study of elliptic and parabolic systems similar to \eqref{pde}. In particular, they focused on systems with quadratic growth. We refer to the book \cite{bensoussan2013regularity} for a collection of results in the elliptic setting, as well as the papers \cite{benelliptic}, \cite{bensoussan}, and \cite{bengame} for other relevant contributions. While these results are related to ours in that they treat systems of PDEs with a gradient non-linearity of quadratic growth, we point out that they are all obtained in the semi-linear case $a = a(t,x)$ and in bounded domains. For the system \eqref{pde2}, it seems that much less is known, and in fact we are not aware of any general solvability result for \eqref{pde2} even in one spatial dimension. 

One motivation for studying \eqref{pde} comes from the theory of forward-backward stochastic differential equations (FBSDEs), which in turn have diverse applications in mathematical finance, stochastic control, stochastic differential game theory, and even stochastic differential geometry. There is a natural link between the PDE system \eqref{pde} and systems of forward-backward stochastic differential equations (FBSDEs) of the form
\begin{align} \label{fbsdeintro}
\begin{cases}
dX_t = H(t,X_t,Y_t,Z_t) dt + \Sigma(t,X_t,Y_t) dB_t, \vspace{.1cm} \\
dY_t = - F(t,X_t,Y_t,Z_t) dt + Z_t dB_t, \vspace{.1cm} \\
X_0 = x_0, \,\, Y_T = G(X_T).
\end{cases}
\end{align}
Here $B$ is a Brownian motion, the data consists of appropriate functions $H$, $\Sigma$, $F$, $G$, and the solution is a triple of adapted processes $(X,Y,Z)$. Such FBSDEs have been studied extensively - we refer to \cite{mayongnotes} or \cite{Zhang} for an introduction to the subject, and to \cite{ma1994,Yong1999LinearFS, Yonglinear2,  mazhangweak2, mazhangweak,   mazhangunified} and the references therein for other significant contributions.
The connection between \eqref{fbsdeintro} and \eqref{pde} is that, roughly speaking, regular enough solutions to PDEs of the form \eqref{pde} allow one to deduce existence results for \eqref{fbsdeintro} - this strategy has been used in many settings since the introduction of the ``four-step scheme" by Ma, Protter and Yong in \cite{ma1994}. When the data of \eqref{fbsdeintro} satisfies standard Lipschitz conditions, this strategy has been used to obtain global existence and uniqueness results for \eqref{fbsdeintro} in \cite{delarue}.

Let us recall in more detail how to solve the FBSDE \eqref{fbsdeintro} via the PDE \eqref{pde}. The idea is to suppose that we find a smooth solution to \eqref{pde} with data $\sigma = \Sigma$, $g = G$, and
\begin{align} \label{translate}
    f^i(t,x,u,p) = F^i(t,x,u,\Sigma(t,x,u)p) + p^i \cdot H(x,u,\Sigma(t,x,u) p),
\end{align}
then It\^o's formula shows that we can produce a solution $(X,Y,Z)$ to \eqref{fbsdeintro} by first solving 
\begin{align*}
    X_t = x_0 + \int_{0}^t H\big(t,X_t,u(t,X_t), \Sigma(t,X_t,u(t,X_t))Du(t,X_t)\big) dt  + \int_0^t \Sigma\big(t,X_t,u(t,X_t)\big) dB_t,
\end{align*}
and then setting $Y_t = u(t,X_t)$, $Z_t = \Sigma(t,X_t,u(t,X_t)) Du(t,X_t)$.

\subsection{Related literature and motivation}

In terms of the FBSDE \eqref{fbsdeintro}, the present work sits at the intersection of three mathematical challenges:
\begin{enumerate}
    \item the quadratic growth of $F$ 
    \item the fact that $n > 1$, i.e. $Y$ is multidimensional (and hence approaches based on the comparison principle fail)
    \item the strong coupling between the equations (i.e. the fact that $\Sigma$ depends on $y$)
\end{enumerate}
Each of these issues has received significant attention in the literature, and it would be impossible to give a thorough literature review for all three. Instead, we simply note that while one-dimensional quadratic BSDEs (i.e. decoupled FBSDEs) were given a thorough treatment in the seminal paper of Kobylanski \cite{Kobylanski}, global existence for quadratic BSDE systems has been considered a central open question for several decades, as noted by Peng in \cite{Pen99}. A breakthrough for systems came in the recent paper of Xing and \v{Z}itkovi\'c \cite{xing2018} (see also \cite{harter2019}, \cite{HuTan16}, and \cite{jackson2021existence} for related contributions in the non-Markovian setting). When all three of the difficulties listed above are present, we are not aware of any existence results even when $T$ is small - the results of \cite{fromm2013existence},  \cite{Luo2017SolvabilityOC} and \cite{kupperluo} do contain results for coupled quadratic FBSDEs with multi-dimensional $Y$, but the results require $\Sigma$ to be independent of $y$ (and even of $x$). 

We now highlight three papers which are especially related to the present work, namely \cite{Delarue2003}, \cite{xing2018}, and \cite{harter2019}. In \cite{xing2018}, H\"older estimates and existence results are obtained in the \textit{semilinear} quadratic case, i.e. the case $\sigma = \sigma(t,x)$ does not depend on $u$ but $f$ has quadratic growth in $p$. In particular, it is shown that $\linf$ estimates on $u$ lead to H\"older estimates on $u$ as soon as the quadratic driver $f$ admits a ``Lyapunov function" - see Theorem 2.5 there. Theorem 2.6 in \cite{harter2019} shows how to strengthen the estimates from \cite{xing2018}, in particular obtaining a gradient estimate (i.e. an estimate on $\norm{Du}_{\linf}$) when the data is smooth enough (still in the semi-linear case). We note that in the semi-linear case, an a-priori estimate of $\norm{Du}_{\linf}$ yields automatically an existence result for classical solutions to the PDE system, thanks to the fact that \eqref{pde} is well-understood when $f$ is Lipschitz (see e.g. Lemma 2.2 of \cite{HU200093}), though this argument does not seem to have appeared in the literature until the recent note \cite{jackson2021note} which studies the FBSDE \eqref{fbsde} in the semi-linear setting. In the quasi-linear case $\sigma = \sigma(t,x,u)$, H\"older and gradient estimates have been obtained via probabilistic arguments in \cite{Delarue2003} for equations corresponding to FBSDEs with Lipschitz coefficients. 

The motivation for understanding quadratic growth comes from the fact that it appears naturally in a variety of applications, for example stochastic differential games, the construction of martingales on Riemannian manifolds, and the existence of equilibria in incomplete financial markets. We refer the reader to Section 3 of \cite{xing2018}, where all three of these examples are discussed. In order to treat stochastic differential games (with uncontrolled volatility) in the more natural \textit{strong formulation}, rather than the weak formulation typically studied through BSDEs, one must solve an FBSDE of the form \eqref{fbsde} with $F$ having quadratic growth (albeit with $\Sigma$ independent of $Y$). See the recent note \cite{jackson2021note}, where this strategy is executed by relying on a-priori estimates from \cite{xing2018}. The motivation for the present paper is to develop a new approach for quadratic FBSDEs which is flexible enough to allow $\Sigma$ to depend on $y$, or equivalently to allow the corresponding PDE to have a non-linearity in the Hessian term. Even in the case that $\Sigma$ does not depend on $y$, however, the approach we develop here still has merit, since it replaces the analytical arguments of \cite{xing2018} and \cite{jackson2021note} (which borrow heavily from the strategy of particular proof strategy of \cite{bengame}) with purely probabilistic (and arguably simpler) methods based on the Krylov-Safonov estimates and the theory of BMO martingales.

The motivation for studying the PDE \eqref{pde2}, meanwhile, comes largely from the link between PDEs of the form \eqref{pde2} and FBSDEs of the form
\begin{align} \label{fbsdeoned}
    \begin{cases}
    dX_t = H(t,X_t,Y_t,Z_t)dt + \Sigma(t,X_t,Y_t,Z_t) dB_t, \vspace{.1cm} \\
    dY_t^i = - F(t,X_t,Y_t,Z_t) + Z_t dB_t, \vspace{.1cm} \\
    X_0 = x_0, \,\, Y_T = G(X_T),
    \end{cases}
\end{align}
which appear in particular when the maximum principle is applied to stochastic control problems or stochastic differential games with controlled volatility. FBSDEs of the type \eqref{fbsdeoned} with $\Sigma$ depending on $z$ are notoriously challenging, and they have been successfully treated primarily under a variety of restrictive monotonicity conditions (see e.g. \cite{hupengfbsde}). Our results on the PDE \eqref{pde2} suggest that it might be possible to obtain positive results for the FBSDE \eqref{fbsde} using non-degeneracy of $\Sigma$ instead of monotonicity, but there is an important hurdle still to clear in order to execute this strategy, see Remark \ref{rmk:oned} below.

\subsection{Our results} \label{subsec:ourresults}

In the case of the equation \eqref{pde}, our main results are a H\"older estimate (Theorem \ref{thm.holder}), a gradient estimate (Theorem \ref{thm:gradest}) and existence results (Theorems \ref{thm:existence} and \ref{thm:existencedecoup}) for \eqref{pde} under appropriate technical and structural conditions on the data $f$, $a = \frac{1}{2} \sigma \sigma^T$ and $g$. We refer to subsection \ref{subsec.assump1} for precise statements of all the hypotheses related to the equation \eqref{pde}. For the H\"older estimates, the main structural condition on $f$ is Hypothesis \ref{hyp.quad}, which asserts the existence of constants $C_f > 0, \epsilon \in (0,1)$ such that 
\begin{align*}
    |f^i(t,x,u,p)| \leq C_f(1 + |p^i||p| + \sum_{j < i} |p^j|^2 + |p|^{2 - \epsilon}), \quad i = 1,...,n. 
\end{align*}
This structural condition is adapted from the conditions appearing in \cite{bensoussan} and \cite{xing2018}, and in that sense our H\"older estimate can be viewed as a generalization of the estimates in \cite{bensoussan} and \cite{xing2018} to the quasi-linear setting. To prove the H\"older estimate in the quasi-linear case $\sigma = \sigma(t,x,u)$, it suffices to prove a H\"older estimate for the semi-linear case $\sigma = \sigma(t,x)$, so long as the estimate depends only on the ellipticity constants of $\sigma$ (and not the regularity of $\sigma$). This is the approach we take. We note that the H\"older estimate in \cite{xing2018} uses the Lipschitz regularity of $\sigma$ (in particular when Aronson's estimate is invoked), and so cannot be applied in the quasi-linear setting. Meanwhile the H\"older estimate in \cite{Delarue2003} is independent of the regularity of $\sigma$, as required, but the argument does not easily adapt to the quadratic case considered here. Our argument for H\"older regularity is similar in spirit to the one in \cite{Delarue2003}, in the sense that we combine tools from the theory of $\BMO$-martingales with the Krylov-Safonov estimates, but the execution is different. In particular, to overcome the quadratic growth we use the concept of sliceability together with the structural condition \ref{hyp.quad} to execute an inductive argument - first showing $u^1$ is H\"older, then showing how this implies that $u^2$ is H\"older, and so on.

After obtaining a global H\"older estimate, we show that it can be used to obtain a gradient estimate when we assume some additional regularity on $f$ (see Hypothesis \ref{hyp.quad2}) in addition to the structural condition \ref{hyp.quad}. The starting point here is to show that the H\"older estimate implies an estimate on the sliceability in $\bmo$ of the $Z$-component of the stochastic representation of $u$. This fact has been observed already in Proposition 5.2 in \cite{xing2018}, but is used in a novel way here. In particular, we study a BSDE representation of the gradient $Du$, and use results from \cite{jackson2021existence} (see also \cite{harter2019} and \cite{Delbaen-Tang}) on linear BSDEs with $\bmo$ coefficients to get the desired gradient estimate. As a corollary of our a-priori estimates, we obtain existence results for \eqref{pde} (see Theorems \ref{thm:existence} and \ref{thm:existencedecoup}). Theorem \ref{thm:existence} gives the existence of classical solutions under sufficient regularity of the data, while \ref{thm:existencedecoup} gives the existence of a ``decoupling solution" (defined below) when the data is less regular. We summarize the results obtained for \eqref{pde} in Table 1 below. \begin{table}[ht] \caption{Summary of estimates for \eqref{pde}}
  \begin{center}
   \label{tab:table1}
   \begin{tabular}{l|l|l} 
      \textbf{Hypotheses} & \textbf{Implication} & \textbf{Precise Statement}\\
      \hline
      \small{\ref{hyp.sigma}} \text{ and } \ref{hyp.ab} & \small{bound on $\norm{u}_{\linf}$} & \small{Lemma \ref{lem.ab}}\\
      \small{\ref{hyp.sigma} \text{ and } \ref{hyp.quad}} & \small{bound on $\norm{u}_{\linf} \implies $ bound on $\norm{u}_{\calpha}$} & \small{Theorem \ref{thm.holder}}\\
      \small{\ref{hyp.sigma} and \ref{hyp.quad2}} & \small{bound on $\norm{u}_{\calpha} \implies $ bound on $\norm{Du}_{\linf}$} & \small{Theorem  \ref{thm:gradest}} \\
      \small{\ref{hyp.ab}, \ref{hyp.reg}, $g \in C^{2,\alpha}$} & \small{$\exists$ classical solution} 
      & \small{Theorem \ref{thm:existence}} \\
      \small{\ref{hyp.ab}, \ref{hyp.sigma}, \ref{hyp.quad2}}, $g$ is Lipschitz 
      &
      \small{$\exists$ decoupling solution} 
      &
      \small{Theorem \ref{thm:existencedecoup}}
    \end{tabular}
  \end{center}
\end{table}

We note that our existence result for \eqref{pde} allows us to deduce an existence result for the FBSDE \eqref{fbsdeintro}, see Theorem \ref{thm.fbsde}. In particular, we obtain existence results for \eqref{fbsdeintro} with $F$ of quadratic growth and satisfying certain structural conditions. This seems to be the first global existence result for a system of the form \eqref{fbsdeintro} when $n > 1$ and $F$ has quadratic growth. Indeed, as explained above the results so far obtained for coupled FBSDEs of quadratic growth (even for small-time well-posedness results) typically require that $\sigma$ is independent of $y$, or even independent of $x$ (see e.g. \cite{Luo2017SolvabilityOC} and \cite{kupperluo}). Thus our global existence result is new even in the small-time (meaning $T$ is sufficiently small) regime. 

Our results for the \eqref{pde2} are similar, but apply only when $f$ is Lipschitz in $(x,u,p)$ and in one spatial dimension. Theorem \ref{thm.apriori} gives a-priori estimates in $C^{1,\alpha}$ and $C^{2,\alpha}$ under appropriate regularity condtions, and Theorem \ref{thm.exist2} gives an existence result for classical solution of \eqref{pde}. 

\begin{remark} \label{rmk:oned}
One might guess that our results for \eqref{pde2} should lead to existence results for an FBSDE of the \eqref{fbsdeoned} where $H$, $\Sigma$ and $F$ are Lipschitz in all arguments and $X,B$ are one-dimensional. Unfortunately this is not the case, because while \eqref{fbsdeoned} will (under some additional technical conditions) be connected to a PDE of the form \eqref{pde2}, it will typically not be true that the data $f$, $b$, and $\sigma$ are globally Lipschitz, even if $H$, $\Sigma$, and $F$ are. It would be desirable to extend the existence result for \eqref{pde2} to cover the FBSDE \eqref{fbsdeoned} in some generality, but we leave this interesting question to future work. 
\end{remark}

\subsection{Organization of the paper}

In the remainder of the introduction we fix notations and conventions. In section \ref{sec.prelim} we discuss some preliminaries, mostly related to $\bmo$ processes and sliceability. Section \ref{sec.apriori} states our main assumptions and results. In section \ref{sec.apriori}, we prove the main a-priori H\"older and gradient estimates for \eqref{pde}. Section \ref{sec.apriori2} contains a-priori estimates for \eqref{pde2}. Finally, Section \ref{sec.existence} contains the proofs of the existence results for the PDEs \eqref{pde} and \eqref{pde2} and the FBSDE \eqref{fbsdeintro}. 
\subsection{Notation and conventions} \label{subsec.notation}

\subsubsection{The probabilistic set-up}
Throughout the paper we fix a probability space $(\Omega, \sF, \bP)$ which hosts a $d$-dimensional Brownian motion $B$. We also fix a time horizon $T \in (0,\infty)$, and $n \in \N$ which will denote the dimension of the unknown process $Y$. The augmented filtration of $B$ is denoted by $\mathbb{F} = (\sF_t)_{0 \leq t \leq T}$. 

\subsubsection{Conventions regarding multidimensional functions and processes} Given $u = (u^i)_{i = 1,...,n} : [0,T] \times \R^d \to \R^n$, we view the spatial gradient $Du$ as an element of $(\R^d)^n$, whose $i^{th}$ element $(Du)^i$ is the gradient $Du^i$ of $u^i$. Similarly, we will at times work with stochastic process $Z$ taking values in $(\R^d)^n$, so the $i^{th}$ element $Z^i$ takes values in $\R^d$. When manipulating elements of $(\R^d)^n$, we interpret multiplication element-wise unless otherwise noted. For example, if $p \in (\R^d)^n$ and $Q \in \R^{d \times d}$, $Q p$ would denote the element of $(\R^d)^n$ whose $i^{th}$ element is $Q p^i \in \R^d$. Similarly, if $p \in (\R^d)^n$ and $q \in \R^d$, then $p q$ would denote the element of $\R^n$ whose $i^{th}$ element is $p^i \cdot q$. This philosophy is used in particular when interpreting the symbol $Z dB$, with $Z$ a process taking values in $(\R^d)^n$. We note here also that we will use $| \cdot |$ to denote the Euclidean norm in any finite-dimensional Euclidean space.

\subsubsection{Universal constants} We view $n$, $d$, $T$ as fixed universal constants. We will use symbols like $C$ to denote a generic constant which can change from line to line. Such a constant may always depend implicitly on the universal constants $n$, $d$, and $T$, but any other dependencies will be made explicit. For example, $\const{D}$
would indicate that $C$ is a constant which depends on the $D$ as well as possibly on the universal constants $n$, $d$, and $T$.

\subsubsection{Spaces of functions}

We will work frequently in parabolic H\"older spaces, so we explain in detail our notations. Fix $\alpha \in (0,1]$. For a function $v  = v(t,x) : [0,T] \times \R^d \to E$, $E$ being some Euclidean space with norm $| \cdot |$ we define the H\"older seminorm 
\begin{align*}
    [v]_{\calpha} = [v]_{\calpha([0,T] \times \R^d)} = \sup_{t \neq t', x \neq x'} \frac{|v(t,x) - v(t',x')|}{|t - t'|^{\alpha/2} + |x - x'|^{\alpha}}, 
\end{align*}
and $\calpha = \calpha([0,T] \times \R^d)$ denotes the functions whose H\"older norm 
\begin{align*}
    \norm{v}_{\calpha} = \norm{v}_{\linf} + [v]_{\calpha}
\end{align*}
is finite. We define $C^{1,\alpha}$ to be the set of $u \in \calpha$ with spatial gradient $Du \in \calpha$, and $C^{2,\alpha}$ to be the set of $u \in \calpha$ with time derivative $\partial_t u \in \calpha$ and spatial gradient and Hessian $Du, D^2u \in \calpha$. We endow $C^{1,\alpha}$ and $C^{2,\alpha}$ with the usual norms 
\begin{align*}
    &\norm{u}_{C^{1,\alpha}} = \norm{u}_{\calpha} + \norm{Du}_{\calpha}, \\
    &\norm{u}_{C^{2,\alpha}} = \norm{u}_{\linf}  + \norm{Du}_{\linf} + \norm{\partial_t u}_{\calpha} + \norm{D^2u}_{\calpha}.
\end{align*}
We will at times also use H\"older norms on $[0,t_0] \times \R^d$, $t_0 < T$, for which we will use obvious notations, e.g. for $u : [0,T] \times \R^d \to \R$,
\begin{align*}
    \norm{u}_{\calpha([0,t_0] \times \R^d)} = \norm{u}_{\linf([0,t_0] \times \R^d)} + \sup_{0 \leq t,t' \leq t_0, t \neq t', x \neq x'} \frac{|v(t,x) - v(t',x')|}{|t - t'|^{\alpha/2} + |x - x'|^{\alpha}}.
\end{align*}
 We indicate local versions of these spaces in a natural way using a subscript. In particular, $C^{2,\alpha}_{\text{loc}}([0,T] \times \R^d)$ will denote the space of functions $u = u(t,x)$ such that for each bounded open set $U \subset \R^d$, $\|u\|_{C^{2,\alpha}([0,T] \times U)} < \infty$. We will say that $u^k \to u$ in $C^{2, \alpha}_{\text{loc}}([0,T] \times \R^d)$ if for each bounded open set $U \subset \R^d$, $\|u - u^k\|_{C^{2,\alpha}([0,T] \times U)} \to 0$. 
 
We define the H\"older spaces of functions defined on $\R^d$ in the same way, i.e. for $g : \R^d \to \R$, 
\begin{align}
\norm{g}_{\calpha} = \sup_{x \neq x'} \frac{|g(x) - g(x')|}{|x - x'|^{\alpha}}, 
\end{align}
and similarly for $\norm{g}_{C^{k, \alpha}}$, $k = 1,2$.

Given an open subset $U$ of $[0,T] \times \R^d$, we say that $v \in C^{1,2}(U)$ if $\partial_t v$, $Dv$, $D^2v$ exist and are continuous on $U$.

\subsubsection{Notions of solutions}

First, recall that any classical solution to \eqref{pde} is expected to be a ``decoupling field" for the FBSDE 
\begin{align} \label{fbsdedef}
    \begin{cases} 
    dX_t = \sigma(t,X_t, Y_t) dB_t, \\
    dY_t = - f(t,X_t,Y_t,\sigma^{-1}(t,X_t,Y_t)Z_t) dt + Z_t dB_t.
    \end{cases}
\end{align}
This allows us to define a probabilistic notion of solution to the \eqref{pde} as follows. A bounded and continuous function $u : [0,T] \times \R^d \to \R^n$ is said to be a \textbf{decoupling solution} of \eqref{pde} if $Du$ is bounded and continuous on $[0,T) \times \R^d$, for each $t \in [0,T]$ and $x \in \R^d$ there is a unique solution $X^{t,x}$ of the SDE  
\begin{align} \label{sde}
    X_{t'}^{t,x} = x + \int_{t}^{t'} \sigma(s,X^{t,x}_s, u(s, X^{t,x}_s))  dB_s
\end{align}
and with $(Y^{t,x},Z^{t,x}) \coloneqq \big(u(\cdot,X^{t,x}), Du(\cdot,X^{t,x})\big)$ we have 
\begin{align} \label{fbsdedecoup}
   \begin{cases} X_{t'}^{t,x} = x + \int_{t}^{t'} \sigma(s,X_s^{t,x},Y_s^{t,x}) dB_s, \vspace{.1cm} \\
    Y_{t'}^{t,x} = g(X_T^{t,x}) + \int_{t'}^T f(s,X_s^{t,x},Y_s^{t,x}, \sigma^{-1}(s,X_s^{t,x}, Y_s^{t,x})Z^{t,x}_s) ds - \int_{t'}^T Z_s^{t,x} dB_s 
    \end{cases}
\end{align}
on the interval $[t,T]$,

We shall also frequently refer to \textbf{classical solutions} of the PDE \eqref{pde} (or \eqref{pde2}). By this, we mean a function $u = (u^i)_{i = 1,...,n} \in C^{1,2}([0,T) \times \R^d ; \R^n) \cap C([0,T] \times \R^d ; \R^n)$ such that 
\begin{enumerate}
    \item $u$ and the spatial gradient $Du$ are bounded on $[0,T] \times \R^d$
    \item the equation \eqref{pde} (or \eqref{pde2}) holds pointwise in $[0,T) \times \R^d$ \item $u(T,x) = g(x)$, for $x \in \R^d$.
\end{enumerate} 
With this definition in place, it is standard to check via It\^o's formula that if $u$ is a classical solution, then $u$ is a decoupling solution, at least provided some minimal regularity on $\sigma$ (see e.g. \eqref{hyp.sigma} below).

\subsubsection{Spaces of processes}
For $1 \leq p \leq \infty$, $\lpee$ denotes the space of $p$-integrable
  $\sF_T$-measurable random variables (taking values in some Euclidean space). We indicate measurability with respect to a sub-$\sigma$-algebra when necessary, i.e. for a sub-$\sigma$-algebra $\sG$ of $\sF_T$, $\lpee(\sG)$ denotes the set of all $\sG$-measurable elements of $\lpee$. For $1 \leq p \leq \infty$, $\spee$ denotes the space of all continuous
  processes $Y$ such that
  \begin{align*}
    \norm{Y}_{\sS^p} \coloneqq
    \norm{Y^*}_{\lpee} < \infty \ewhere Y^* = \sup_{0 \leq t \leq T} |Y_t|.
  \end{align*}
  $\BMO$ denotes the space of continuous martingales $M$ such that
  \begin{align*}
    \norm{M}_{\BMO} \coloneqq
    \esssup_{\tau} \norm{\bE_{\tau}[ \abs{M_T - M_{\tau} }^2 ]}_{L^{\infty}}^{\frac{1}{2}}
    < \infty,
  \end{align*}
  where the supremum is taken over all stopping times $0 \leq \tau \leq T$, while $\bmo$ denotes the space of progressive processes $\gamma$ such that
  \begin{align*}
    \norm{\gamma}_{\bmo}^2 \coloneqq \sup_{\tau}
    \bE_{\tau}\left[\int_{\tau}^T \abs{\gamma}^2 ds\right]
    < \infty.
  \end{align*}
 Similarly, $\bmo^{1/2}$ denotes the space of progressive processes $\beta$ such
  that
  \begin{align*}
    \norm{\gamma}_{\bmo^{1/2}} \coloneqq
    \sup_{\tau} \bE_{\tau}\left[\int_{\tau}^T \abs{\gamma}\, ds\right]
    < \infty, \text{ i.e., }
    \norm{\gamma}_{\bmo^{1/2}} = \norm{\sqrt{\abs{\gamma}}}^2_{\bmo}.
  \end{align*}
If necessary, we emphasize the co-domain of the space of processes under consideration, e.g. by writing $\bmo(\R^d)$ for the space of $\bmo$ processes taking values in $\R^d$. We also note that at times we will work with processes defined only on a subinterval $[t_0,T]$ of $[0,T]$. We can extend all the definitions above to such processes in a natural way. In particular, we highlight that if $Y$ is a continuous process defined on $[t_0,T]$, then 
\begin{align*}
     \norm{Y}_{\sS^p} \coloneqq \norm{\sup_{t_0 \leq t \leq T} |Y_t|}_{\lpee}. 
\end{align*}
If $\gamma$ is defined on $[t_0,T]$, we denote by $\norm{\gamma}_{\bmo}$ the quantity $\norm{\tilde{\gamma}}_{\bmo}$, where $\tilde{\gamma}$ denotes the extension of $\gamma$ to $[0,T]$ by $0$: 
\begin{align*}
    \tilde{\gamma}_t = \begin{cases}
    0 & 0 \leq t < t_0, \\
    \gamma_t & t_0 \leq t \leq T. 
    \end{cases}
\end{align*}
Finally, in an abuse of notation $\linf$ denote also the set of progressively measurable processes $Z$ with $\norm{Z}_{\linf} = \esssup_{t,\omega} |Z_t(\omega)| < \infty$.

\section{Assumptions and main results} \label{sec.results}

\subsection{Assumptions related to \eqref{pde}} \label{subsec.assump1}
The data for \eqref{pde} consists of the three functions $\sigma$, $f$, and $g$, where 
\begin{align*} 
    & \sigma = \sigma(t,x,u) = (\sigma^{jk}(t,x,u))_{j,k = 1,...,d} : [0,T] \times \R^d \times \R^n \to \R^{d \times d}, \nonumber \\
    & f = f(t,x,u,p) = (f^i(t,x,u,p))_{i = 1,...,n} : [0,T] \times \R^d \times \R^n \times (\R^d)^n \to \R^n, \text{ and } \nonumber \\
    & g = g(x) = (g^i(x))_{i = 1,...,n} : \R^d \to \R^n. 
\end{align*}
We now state the assumptions which will be made at various points on the data. The first assumption concerns the regularity and non-degeneracy of the matrix $\sigma$.
\be \label{hyp.sigma} \tag{$H_{\sigma}$} \begin{cases}
\text{The matrix $\sigma$ is symmetric, and there are constants } L_{\sigma}, C_{\sigma} > 0 \text{ such that } \sigma \text{ satisfies the estimates }  \\
\hspace{.5cm} 1) \,\, |\sigma(t,x,u) - \sigma(t,x',u')| \leq L_{\sigma} \big(|x - x'| + |u - u'|\big), \\
\hspace{.5cm} 2) \,\,
\frac{1}{C_{\sigma}} |z|^2 \leq |\sigma(t,x,u) z|^2 \leq C_{\sigma} |z|^2, \\
\text{for all } t \in [0,T], \,\, x,x' \in \R^d, \,\, u, u' \in \R^n, \,\, z \in \R^d. 
\end{cases}
\ee 
For a general quadratic $f$, a-priori estimates on $\norm{u}_{\linf}$ may not be possible, but there are many structural conditions on $f$ for which such estimates are known to hold. We give two such conditions here. The first is adapted from \cite{xing2018}, and the other one, which is simple to prove, allows us to cover the case studied in \cite{delarue}. We emphasize that the H\"older and Lipschitz estimates proved below do not require the conditions \ref{hyp.ab1} or \ref{hyp.ab2}, given below, which are used only to obtain estimates on $\norm{u}_{\linf}$.

\be \label{hyp.ab1} \tag{$H_{\text{AB1}}$} \begin{cases}
\text{The driver $f$ can be written as $f^i(t,x,u,p) = p^i \cdot b_0(t,x,u,p) + b^i(t,x,u,p)$}, \\ \text{where $b_0$ and $(b^i)_{i = 1,...,n}$ satisfy} \\
\hspace{.5cm} 1) \,\, |b_0(t,x,u,p)| \leq M\big( 1 + \kappa(|u| + |p|)\big) \\
\hspace{.5cm} 2) \text{ $
a_q^T b(t,x,u,p) \leq M + \frac{1}{2} \abs{a_q^T p}^2$}, \\
\text{for all $(t,x,u,p) \in \all$ and $q = 1,...,Q$,} \\
\text{and for some constant $M > 0$, increasing function $\kappa : \R^+ \to \R^+$,}\\
\text{and set $\{a_1,...,a_Q\}$ of vectors positively spanning $\R^n$.}
\end{cases}
\ee 
\be \label{hyp.ab2} \tag{$H_{\text{AB2}}$} \begin{cases}
\text{The driver $f$ can be written
$f^i(t,x,u,p) = p^i \cdot b_0(t,x,u,p) + b^i(t,x,u,p)$}, \\
\text{where $b_0$ and $(b^i)_{i = 1,...,n}$ satisfy }
\\
\hspace{.5cm} 1) \,\, |b_0(t,x,u,p)| \leq M\big( 1 + \kappa(|u| + |p|)\big),\\
\hspace{.5cm} 2) \,\,|b^i(t,x,u,p)| \leq M\big(1 + |u| + |p|\big), \\
\text{for some $M > 0$, some increasing function } \\ \text{$\kappa : \R_+ \to \R_+$ and for all } (t,x,u,p) \in \all
\end{cases}
\ee 
For simplicity, we put these two a-priori boundedness conditions together as follows: 
\be \label{hyp.ab} \tag{$H_{\text{AB}}$}
\text{Either \ref{hyp.ab1} or \ref{hyp.ab2} hold.}
\ee 

\begin{remark}
We discuss briefly how the conditions \eqref{hyp.ab1} and \eqref{hyp.ab2} lead to $\linf$ estimates on $u$. Firstly, the term $p^i \cdot b_0(t,x,u,p)$ can typically be safely ignored when searching for $\linf$ bounds. The analytical explanation for this is that it can be viewed as part of the linear operator being applied to each $u^i$ in the equation \eqref{pde}, for example under \eqref{hyp.ab1} we can rewrite the PDE \eqref{pde} as 
\begin{align*}
\partial_t u^i + \mathcal{L}(u,Du)(u^i) + b^i(t,x,u,Du) = 0, 
\end{align*}
with $\mathcal{L}(u,Du)$ denoting the differential operator 
\begin{align*}\mathcal{L}(u,Du)(v) = \tr( a(t,x,u) D^2 v) + b_0(t,x,u,Du) \cdot Dv.
\end{align*}
The corresponding probabilistic explanation is that in the corresponding FBSDE, the term coming from $p^i \cdot b_0(t,x,u,p)$ can be essentially removed through the Girsanov transformation (see the proof of Lemma \ref{lem.ab}). 

Meanwhile, the conditions placed on $b^i$ in \eqref{hyp.ab1} are borrowed largely from the ``a-priori boundedness condition" in \cite{xing2018}, which was in turn inspired by similar conditions in the literature on parabolic systems, see e.g. \cite{bengame}. Roughly speaking, it allows to obtain $\linf$-estimates for the system by showing that one-dimensional projections of the solution $u$ (along the directions $a_q \in \R^n$) are (approximately) sub-solutions of (scalar) PDEs, and then employing the comparison principle to get $\linf$ estimates on $u$ in each of the directions $a_q$. 

Finally, the condition on $b^i$ appearing in \eqref{hyp.ab2} is fairly easy to explain - it is a linear growth assumption which ensures that (after a Girsanov transformation handles the term coming from $p^i \cdot b_0$) the BSDE representing $u$ can be estimated by standard methods. We again refer to the proof of Lemma \ref{lem.ab} for more details.
\end{remark}

The H\"older estimates on $u$ will be obtained under the following structural condition on the quadratic driver $f$. We follow \cite{xing2018} in calling it a ``Bensoussan-Frehse" condition, because of the resemblance to the structural condition used in \cite{bensoussan}. 

\be \label{hyp.quad} \tag{$H_{BF}$}
\begin{cases}
\text{There are constants $C_f > 0$, $0 < \epsilon < 1$ such that $f$ satisfies} \\
\hspace{.5cm} |f^i(t,x,u,p)| \leq C_f \big(1 + |p^i||p| + \sum_{j < i} |p^j|^2 + |p|^{2 - \epsilon}\big), \\
\text{for all } (t,x,u,p) \in \all, \,\, i = 1,...,n.
\end{cases}
\ee
\begin{remark} \label{rmk.bf}
We note that following a computation in \cite{bensoussan}, one can show that if \ref{hyp.quad} holds, then there are measurable functions $h^i = h^i(t,x,u,p) : \all \to \R^d$, $k^i = k^i(t,x,u,p) : \all \to \R$
such that
\begin{align} \label{bfsuff}
    f^i(t,x,u,p) = p^i \cdot h^i(t,x,u,p) + k^i(t,x,u,p)
\end{align}
and the estimates 
\begin{align} \label{bfsuffest}
    &|h^i(t,x,u,p)| \leq C_Q(1 + |p|), \quad
    |k^i(t,x,u,p)| \leq C_Q(1 + \sum_{j < i} |p^j|^2 + |p|^{2 - \epsilon})
\end{align}
hold for all $(t,x,u,p) \in \all$.
Indeed, taking 
\begin{align*}
    &h^i(t,x,u,p) = \frac{f^i(t,x,u,p)}{\big(1 + |p^i| |p| + \sum_{j < i} |p^j|^2 + |p|^{2 - \epsilon}\big)} \frac{ p^i |p|}{|p^i|} 1_{|p^i| \neq 0}, \text{ and } \\
    &k^i(t,x,u,p) = \frac{f^i(t,x,u,p)}{\big(1 + |p^i| |p| + \sum_{j < i} |p^j|^2 + |p|^{2 - \epsilon}\big)} (1 + \sum_{j < i} |p^j|^2 + |p|^{2 - \epsilon}), 
\end{align*}
it is easy to check that the estimates in \eqref{bfsuffest} hold. 
\end{remark}

To bootstrap from H\"older to gradient estimates, we will need some regularity of the coefficients in addition to the growth condition. The following condition states that $f$ is locally Lipschitz in $(x,u,p)$, with a Lipschitz constant depending on $|p|$ in a natural way.

\be \label{hyp.quad2} \tag{$H_{\text{Q}}$}
\begin{cases}
\text{In addition to the condition \ref{hyp.quad}, } f \text{ satisfies the estimates}\\ \hspace{.5cm} 1) \,\,
|f(t,x,u,p) - f(t,x',u',p)| \leq C_f \big(1 + |p|^2\big)(|x - x'| + |u - u'|), 
\\ \hspace{.5cm} 2) \,\,|f(t,x,u,p) - f(t,x,u,p')| \leq C_f (1 + |p| + |p'|) |p - p'|,
\end{cases}
\ee
\begin{remark}
Suppose that $f = f(t,x,u,p)$ is continuously differentiable in $(x,u,p)$ for each fixed $t$. Then \ref{hyp.quad2} is equivalent to the estimates \begin{align*}
|D_x f(t,x,u,p)| + |D_u f(t,x,u,p)| \leq C_f \big(1 + |p|^2\big), 
\quad |D_p f(t,x,u,p)| \leq C_f  \big(1 + |p|\big).
\end{align*}
\end{remark}
Finally, to get a classical solution to \eqref{pde}, we will need some H\"older regularity of $\sigma$ and $f$ in time:
\be \label{hyp.reg} \tag{$H_{\text{Reg}}$}
\begin{cases}
\text{In addition to \ref{hyp.sigma} and \ref{hyp.quad2}, we have the estimates }
\\
\hspace{.5cm} 1) \,\, |\sigma(t,x,u) - \sigma(t',x,u)| \leq L_{\sigma} |t - t'|^{\alphat}, \\
\hspace{.5cm} 2) \,\, |f(t,x,u,p) - f(t',x,u,p)| \leq C_f |t- t'|^{\alphat},  \\
\text{for all } t,t' \in [0,T], \,\, (x,u,p) \in \R^d \times \R^n \times (\R^d)^n, \text{ and some } \alpha_0 \in (0,1).
\end{cases}
\ee 

\begin{remark}
To be clear, we have stated the regularity and structure conditions above in such a way that the implications
\begin{align*}
    \text{\ref{hyp.reg}} \implies \text{\ref{hyp.sigma}} \text{ and } \text{\ref{hyp.quad2}}, \quad \text{\ref{hyp.quad2}} \implies \text{\ref{hyp.quad}}, 
\end{align*}
hold. 
\end{remark}

\subsection{Statement of the results for \eqref{pde}}

We now state our results for the equation \eqref{pde}. We begin with an a-priori estimate for $\norm{u}_{\linf}$. 

\begin{lemma} \label{lem.ab}
Suppose that \ref{hyp.sigma} and \ref{hyp.ab} hold. Suppose further that $g$ is bounded. Then for any decoupling solution $u$ of \eqref{pde}, we have 
\begin{align*}
    \norm{u}_{\linf} \leq C, 
\end{align*}
for a constant $C$ depending only on $\norm{g}_{\linf}$, $C_{\sigma}$, and either $\{a_m\}$ and $\rho$ (if we assume \ref{hyp.ab1}) or $M$ (if we assume \ref{hyp.ab2}). 
\end{lemma}

The next result gives an a-priori H\"older estimate for $u$. 

\begin{theorem} \label{thm.holder}
Suppose that \ref{hyp.sigma} and \ref{hyp.quad} hold. Suppose further that $g \in \cbeta$ for some $\beta \in (0,1)$, and that $u$ is a decoupling solution of \eqref{pde}. Then for some $\alpha \in (0,1)$ and $C > 0$ depending on $\beta, \norm{g}_{\cbeta}, C_{\sigma}, C_Q, \epsilon, \norm{u}_{\linf}$, we have 
\begin{align*}
    \norm{u}_{\calpha} \leq C.
\end{align*}
\end{theorem}
Our next result is a gradient estimate for \eqref{pde}.
\begin{theorem} \label{thm:gradest}
Assume that \ref{hyp.sigma} and \ref{hyp.quad2} hold.
Suppose further that $\sigma$ is continuously differentiable in $(x,u)$ and $f$ is continuously differentiable in $(x,u,p)$ for each fixed $t$ and that $g \in C^1(\R^d)$ with bounded derivative. Let $u$ be a classical solution of \eqref{pde} with $Du \in C^{1,2}([0,T] \times \R^d)$. Then for any $\alpha \in (0,1)$, we have
\begin{align*}
    \norm{Du}_{\linf} \leq C, \, \, \const{\norm{Dg}_{\linf},C_{\sigma},L_{\sigma},C_Q, \alpha, \norm{u}_{\calpha}}.
\end{align*}
\end{theorem}
Finally, we obtain the following existence results as consequences of our a-priori estimates.
\begin{theorem} \label{thm:existence}
Suppose that \ref{hyp.ab} holds. Suppose also that \ref{hyp.reg} holds. Finally, suppose that $g$ is $C^{2,\beta}(\R^d)$ for some $\beta \in (0,1)$. Then, there is a unique classical solution $u$ to \eqref{pde}, which satisfies $u \in C^{2,\alpha}$ for some $\alpha \in (0,1)$. 
\end{theorem}
If the terminal condition is only Lipschitz, we can still get decoupling solution to \eqref{pde}, and we can also drop the assumption \ref{hyp.reg}.
\begin{theorem} \label{thm:existencedecoup}
Suppose that \ref{hyp.ab} holds. Suppose also that \ref{hyp.sigma} and \ref{hyp.quad2}, and that $\sigma$, $f$ and $g$ are continuous in all arguments. Finally, suppose that $g$ is Lipschitz. Then, there is a unique decoupling solution $u$ to \eqref{pde}. 
\end{theorem}

\begin{remark} \label{rmk.unique}
The uniqueness statement in Theorem \ref{thm:existence} is implied by the uniqueness statement in Theorem \ref{thm:existencedecoup}, since every classical solution to \eqref{pde} is also a decoupling solution. Moreover, we have defined decoupling solutions to be uniformly Lipschitz in space, so that if $u$ and $\tilde{u}$ were two decoupling solutions, then they would both be decoupling solutions to the PDE \eqref{pde}, but with the driver $f$ replaced by the driver $\tilde{f}(t,x,u,p) = f(t,x,u,\pi(p))$ for some smooth cut-off function $\pi$, i.e. $\pi$ is Lipschitz, bounded and $\pi(p) = p$ for $|p| \leq K$ with $K$ sufficiently large. Since $\tilde{f}$ is uniformly Lipschitz in $(x,u,p)$, the uniqueness statement in Theorem \ref{thm:existence} (hence also in Theorem \ref{thm:existencedecoup}) follows easily from the existing literature on Lipschitz FBSDEs (see e.g. \cite{delarue}). 
\end{remark}


\subsection{Assumptions related to \eqref{pde2}} 

Now we state the assumptions which we will use when studying \eqref{pde2}. Recall that in this case the data is
 \begin{align} \label{data.pde2}
    & \sigma = \sigma(t,x,u,p) : [0,T] \times \R \times \R^n \times \R^n \to \R, \nonumber \\
    & f = f(t,x,u,p) = (f^i(t,x,u,p))_{i = 1,...,n}: [0,T] \times \R \times \R^n \times \R^n \to \R^n, \text{ and } \nonumber \\
    & g = g(x) = (g^i(x))_{i = 1,...,n} : \R \to \R^n.
\end{align}
We start with a non-degeneracy and regularity condition for $\sigma$. 
\be \label{hyp.sigma2} \tag{$H^{1}_{\sigma}$} \begin{cases}
\text{There are constants } L_{\sigma},C_{\sigma} > 0 \text{ such that } \sigma \text{ satisfies the estimates }  \\
\hspace{.5cm} 1) \,\, |\sigma(t,x,u,p) - \sigma(t,x',u',p')| \leq L_{\sigma} \big(|x - x'| + |u - u'| + |p - p'|\big), \\
\hspace{.5cm} 2) \,\,
\frac{1}{C_{\sigma}}  \leq |\sigma(t,x,u,p)|^2 \leq C_{\sigma}, \\
\text{hold for all } t \in [0,T], \,\, x,x
\in \R, \,\, u, u' \in \R^n, \,\, p, p' \in \R^n.
\end{cases}
\ee 
Next, we state the appropriate regularity conditions for $f$.
\be \label{hyp.lip} \tag{$H^1_{\text{Lip}}$}
\begin{cases}
\text{There are constants $C_f$ such that the estimates} \\
\hspace{.5cm} 1) \,\, |f(t,x,u,p) - f(t,x',u',p')| \leq C_f (|x - x'| +|u - u'| + |p - p'|\big), \\
\hspace{.5cm} 2) \,\,|f(t,x,u,p)| \leq C_f(1 + |u| + |p|), 
\\ \text{ hold for all } t \in [0,T],\,\, x,x' \in \R, \,\, u, u' \in \R^n, \,\, p, p' \in (\R^d)^n
\end{cases}
\ee

\be \label{hyp.reg2} \tag{$H^1_{\text{Reg}}$}
\begin{cases}
\text{In addition to \ref{hyp.lip} and \ref{hyp.sigma2}, there is a constant $\alphat \in (0,1)$ such that}
\\
\hspace{.5cm} 1) \,\, |\sigma(t,x,u,p) - \sigma(t',x,u,p)| \leq \ct |t - t'|^{\alphat}, \\
\hspace{.5cm} 2) \,\, |f(t,x,u,p) - f(t',x,u,p)| \leq \ct |t- t'|^{\alphat},  \\
\text{hold for all } t,t' \in [0,T], \,\, (x,u,p) \in \R^d \times \R^n \times (\R^d)^n.
\end{cases}
\ee 

\subsection{Statement of the results for \eqref{pde2}}

We start with an a-priori estimate for \eqref{pde2} 

\begin{theorem} \label{thm.apriori}
Suppose that \ref{hyp.sigma2}, and \ref{hyp.lip} hold. Suppose further that $g \in C^{1, \beta}$ for some $\beta \in (0,1)$. Finally, suppose that $u$ is a classical solution to \eqref{pde2} with $Du \in C^{1,2}$ and $D^2u$ bounded. Then, for some $\alpha = \alpha(\beta,\norm{g}_{C^{1,\beta}}, L_{\sigma}, C_{\sigma}, C_f)$, we have
\begin{align*}
    \norm{Du}_{\calpha} \leq C, \,\, \const{\beta,\norm{g}_{C^{1,\beta}}, L_{\sigma}, C_{\sigma}, C_f}.
\end{align*}
If in addition \ref{hyp.reg2} holds and $g \in C^{2, \beta}$, then for some (potentially different) $\alpha = \alpha(\beta,\norm{g}_{C^{1,\beta}}, L_{\sigma}, C_{\sigma}, C_f)$, we have
\begin{align*}
    \norm{u}_{C^{2, \alpha}} \leq C, \,\, \const{\beta,\norm{g}_{C^{2,\beta}}, L_{\sigma}, C_{\sigma}, C_f, \ct, \alphat}, 
\end{align*}
\end{theorem}
This a-priori estimate can be combined with the method of continuity to give the following existence result. 
\begin{theorem} \label{thm.exist2}
Suppose that \ref{hyp.sigma2} and \ref{hyp.reg2} hold. Suppose further that for some $\beta \in (0,1)$, $g \in C^{2,\beta}$. Then, there exists a classical solution $u$ to \eqref{pde2}. 
\end{theorem}

\subsection{Application to the FBSDE \eqref{fbsdeintro}}

We now describe the hypothesis on the data
\begin{align} \label{data.fbsde}
    &H = H(t,x,y,z) : [0,T] \times \R^d \times \R^n \times (\R^d)^n \to \R^d, \nonumber \\ &\Sigma = \Sigma(t,x,y)  : [0,T] \times \R^d \times \R^n \times (\R^d)^n \to \R^{d \times d}, \nonumber \\
    &F = F(t,x,y,z) : [0,T] \times \R^d \times \R^n \times (\R^d)^n \to \R^n, \\ \nonumber
    &G = G(x) : \R^d \to \R^n. 
\end{align}
under which we will obtain existence for \eqref{fbsdeintro}. For $\Sigma$ and $F$, we will essentially borrow the conditions we have already defined for $\sigma$ and $f$ above.

\be \label{hyp.Sig} \tag{$H_{\Sigma}$}
\text{$\Sigma$ is continuous and satisfies \ref{hyp.sigma}}
\ee 

\be \label{hyp.F} \tag{$H_{F}$} 
\text{$F$ is continuous and satisfies \ref{hyp.ab} and \ref{hyp.quad2}}.
\ee 

\be \label{hyp.H} \tag{$H_{H}$} \begin{cases}
\text{$H$ is continuous and there is a constant $C_H > 0$}\\ \text{and an increasing function $\kappa : \R_+ \to \R_+$ such that the estimates} \\
\hspace{.5cm} 1) \,\, |H(t,x,y,z) - H(t,x',y',z)|  \leq C_H\big(1 + |z|\big)(|x - x'| + |y-y'|), \\
\hspace{.5cm} 2) \,\,
|H(t,x,y,z) - H(t,x,y,z')| \leq C_H |z - z'|\\
\hspace{.5cm} 3) \,\, 
|H(t,x,y,z)| \leq C_H(1 + |\kappa(|y|)| + |z|)\\
\text{hold for all } t, \in [0,T], \,\, x,x' \in \R^d, \,\, y,y' \in \R^n, \,\, z,z' \in (\R^d)^n. 
\end{cases}
\ee 
Here is the existence result for \eqref{fbsdeintro}.

\begin{theorem} \label{thm.fbsde}
Suppose that \ref{hyp.Sig}, \ref{hyp.F} and \ref{hyp.H} hold. Suppose further that $G$ is Lipschitz. Then there is a unique solution $(X,Y,Z) \in \stwo \times \sinf \times \linf$ to \eqref{fbsdeintro}. 
\end{theorem}

\begin{remark} \label{rmk.uniquefbsde}
To be clear, Theorem \ref{thm.fbsde} asserts uniqueness in the class $\stwo \times \sinf \times \linf$, which follows easily from results on Lipschitz FBSDEs. It seems natural to expect uniqueness also in the slightly larger class $\stwo \times \sinf \times \bmo$. The standard way to obtain this latter, more general, uniqueness statement would be to first prove existence and uniqueness in $\stwo \times \sinf \times \bmo$ when $T$ is sufficiently small, and then bootstrap this local result with the help of the decoupling solution $u$. Indeed, the arguments introduced in \cite{ma1994} show that as a general rule, 
\begin{align*}
     \big(\exists \text{ smooth solution of PDE }\big) + \big(\text{local well-posedness of FBSDE}\big) \\ \implies \big(\text{global uniqueness of FBSDE}\big).
\end{align*} But unlike in the Lipschitz case, employing the Banach fixed point theorem to get existence and uniqueness in a space like $\stwo \times \sinf \times \bmo$ for $T$ small in the present quadratic case seems relatively challenging - there are some small-time results for quadratic FBSDEs appearing in \cite{Luo2017SolvabilityOC} and \cite{kupperluo}, but none general enough to apply in our setting.
\end{remark}

\section{Preliminaries} \label{sec.prelim}
This section is auxiliary in nature, and contains statements and proofs of a number of results which will be necessary for the proof of the main a-priori estimates in the next section. 

\subsection{The space $\bmo$}

We now recall some basic facts about the space $\bmo$. The important point is that for algebraically compatible $a$, $\norm{a}_{\bmo} = \norm{ \int a dB}_{\BMO}$.

The following Lemma can be deduced from Theorem 3.6 of \cite{Kazamaki}, which explains that a ``bmo change of measure" from $\bP$ to $\bQ$ induces a linear isomorphism from $\BMO(\bP)$ to $\BMO(\bQ)$. \footnote{Actually, Theorem 3.6 of \cite{Kazamaki} implies only the existence of the constant $C$ appearing in Lemma \ref{lem.kaz}, for each $a \in \bmo$. The fact that $C$ can be chosen to depend only on $\norm{a}_{\bmo}$ is clear from Kazamaki's proof.}
\begin{lemma} \label{lem.kaz}
Suppose that $\norm{a}_{\bmo} < \infty$, and define a measure $\bQ$ by $d\bQ = \mathcal{E}(\int a dB)_T d\bP$. Then $\bmo(\bP) = \bmo(\bQ)$, and 
\begin{align*}
    \frac{1}{C} \norm{b}_{\bmo(\bQ)} \leq \norm{b}_{\bmo(\bP)} \leq C \norm{b}_{\bmo(\bQ)},
\end{align*}
for each $b \in \bmo(\bP)$,
and some $C$ depend only on $\norm{a}_{\bmo}$. As a consequence, 
\begin{align*}
    \frac{1}{C} \norm{b}_{\bmoh(\bQ)} \leq \norm{b}_{\bmoh(\bP)} \leq C \norm{b}_{\bmoh(\bQ)}
\end{align*}
for each $b \in \bmoh(\bP)$. 
\end{lemma}

This leads to the following Lemma, which will be key in the proof of the H\"older estimate for \eqref{pde}.

\begin{lemma} \label{lem.com}
Let $a \in \bmo$ and $\bQ$ be defined by $d \bQ = \mathcal{E}(\int a dB)_T d \bP$. Then for $A \in \sF$, we have 
\begin{align*}
    \bQ[A] \geq C \bP[A]^q, 
\end{align*}
for some $C,q > 0$ depending only on $\norm{a}_{\bmo}$. 
\end{lemma}

\begin{proof}
It follows from a computation that 
\begin{align*}
    \frac{d\bP}{d\bQ} = \mathcal{E}(\int - a \cdot d B^a)_T, \quad B^a = B - \int a dt.
\end{align*}
By using Lemma \ref{lem.kaz} together with Theorem 3.1 in \cite{Kazamaki}, we can find $p > 1$, $C > 0$ depending only on $\norm{a}_{\bmo}$ such that 
\begin{align*}
    \norm{\mathcal{E}(\int - a \cdot  dB^a)_T}_{\lpee(\bQ)} \leq C.
\end{align*}
Thus by the H\"older inequality
\begin{align*}
\bP[A] \leq \int 1_A \mathcal{E}(\int - a \cdot dB^a)_T d\bQ \leq \norm{\mathcal{E}(\int a \cdot dB)_T}_{\lpee(\bQ)} \norm{1_A}_{\lque(\bQ)} 
\leq C \bQ[A]^{1/q}, 
\end{align*}
where $q$ is the conjugate exponent of $p$ This completes the proof.
\end{proof}

The next lemma states simply that if $|a|^{1 + \epsilon} \in \bmo$, then $a$ is sliceable (see subsection \ref{subsec.slice} below for the definition). 
\begin{lemma} \label{lem.slice}
Suppose that for some $\epsilon > 0$, $\norm{|a|^{1 + \epsilon}}_{\bmo} < \infty$. Then, for any constants $t$, $\delta$ such that 
\begin{align*}
    0 \leq t - \delta \leq t \leq T, 
\end{align*}
we have 
\begin{align*}
    \norm{a 1_{[t-\delta,t]}}_{\bmo} \leq C\delta^{\alpha}, 
\end{align*}
where $\alpha = \frac{\epsilon}{1 + \epsilon}$, $C = C(\norm{|\alpha|^{1 + \epsilon}}_{\bmo})$.
\end{lemma}

\begin{proof}
Let $\tau$ be a stopping time. For simplicity, set $\sigma = \big(\tau \vee (t- \delta) \big) \wedge t$. Notice that 
\begin{align*}
    \bE_{\tau}[\int_{\tau}^T 1_{[t-\delta,t]} |a|^2 ds ] = \bE_{\tau}[\int_{\sigma}^{t} |a^2|ds] = \bE_{\tau}\big[ \bE_{\sigma} \int_{\sigma}^{t} |a|^2 ds ]\big],
\end{align*}
so 
\begin{align*}
    \norm{ \bE_{\tau}[\int_{\tau}^T 1_{[t-\delta,t]} |a|^2 dt ]}_{\linf} \leq \norm{\bE_{\sigma }[\int_{\sigma}^{t} |a|^2 ds}_{\linf}. 
\end{align*}
Since 
\begin{align*}
    \bE_{\sigma }[\int_{\sigma}^t |a|^2 ds] \leq  \big(\bE_{\sigma}[\int_{\sigma}^{t} |a|^{2 + 2\epsilon} ds]\big)^{\frac{1}{1 + \epsilon}} \big(\bE_{\sigma}[\int_{\sigma}^{t} 1 ds]\big)^{\frac{\epsilon}{1 + \epsilon}}
    \leq C \delta^{\frac{\epsilon}{1 + \epsilon}},
\end{align*}
we can conclude.
\end{proof}

\subsection{Sliceability and linear BSDEs with $\bmo$ coefficients} \label{subsec.slice}

We now gives some additional preliminaries concerning the concept of sliceability, and linear BSDEs with $\bmo$ coefficients. These ideas are taken largely from \cite{jackson2021existence}. 
First, we define a \textbf{random partition} of $[0,T]$ a a collection $\ption$  of
    stopping times such that $0=\tau_0 \leq \tau_1 \leq \dots \leq \tau_m=T$.
    The set of all random partitions is denoted by $\sP$. For $A\in\bmo$, the \textbf{index of sliceability} for $A$ is
    the function $N_{A}:(0,\infty)\to \N\cup\set{\infty}$ defined as follows.
    For $\delta>0$, $N_{A}(\delta)$ is the smallest natural number $m$ such
    that there exists a random partition $\ption \in \sP$ such that
    \begin{align}
      \label{slic}
      \norm{A \itauk}_{\bmo} \leq \delta \eforall 1\leq k \leq m.
    \end{align}
    If no such $m$ exists, we set $N_{A}(\delta)=\infty$. A bmo-process $A$ is said to be \textbf{$\delta$-sliceable} if
    $N_{A}(\delta)<\infty$ and \textbf{sliceable} if it is $\delta$-sliceable
    for each $\delta>0$. A family $\sA \subseteq \bmo$ is said to be \textbf{uniformly
    sliceable} if
    \begin{align*}
      \sup_{A\in \sA} N_{A}(\delta)<\infty \eforall \delta>0.
    \end{align*}
  Sliceability and the related notions given above are defined for the space
  $\bmoh$ in the same way.

Consider now the linear BSDE 
\begin{align} \label{linbsde}
   Y_t = \xi + \int_t^T \big( \alpha_s Y_s + A_s Z_s + \beta_s \big) ds - \int_t^T Z_s dB_s, 
\end{align}
or, unwrapping the conventions on multi-dimensional processes introduced above,
\begin{align} \label{linbsdecomp}
   Y_t^i = \xi^i + \int_t^T \big( \alpha_s^i \cdot Y_s + \sum_{j = 1}^n A^{ij}_s \cdot Z^j_s + \beta_s^i \big) ds - \int_t^T Z^i_s dB_s. 
\end{align}
The data for this problem is 
\begin{align*}
    &\alpha = (\alpha^i)_{i = 1,...,n} \in \bmoh((\R^n)^n), \quad
    A = (A^{ij})_{i,j = 1,...,n} \in \bmo( (\R^d)^{n \times n}), \\
    &\beta = (\beta^i)_{i = 1,...,n} \in \bmoh(\R^n), \quad \xi = (\xi^i)_{i = 1,...,n} \in \linf(\R^n)
\end{align*}
and the solution is a pair of processes
\begin{align*}
    Y = (Y^i)_{i = 1,...,n} \in \sinf(\R^n), \quad
    Z = (Z^i)_{i = 1,...,n} \in \bmo((\R^d)^n)
\end{align*}
satisfying \eqref{linbsde} a.s., for each $t \in [0,T]$. The following is a consequence of Theorem 2.9 of \cite{jackson2021existence}, tailored to our setting. 

\begin{proposition} \label{prop:sliceableeqns}
Suppose that $A$ and $\alpha$ are sliceable, in the sense that
\begin{align*}
    N_{\alpha}(\delta) + N_{A}(\delta) \leq K(\delta),
\end{align*}
for some $K : (0,\infty) \to \N$. Then, for each $(\beta, \xi) \in \bmoh \times \linf$, there is a unique solution to \eqref{linbsde} satisfying 
\begin{align*}
   \norm{Y}_{\linf} + \norm{Z}_{\bmo} \leq C\big( \norm{\xi}_{\linf} + \norm{\beta}_{\bmoh} \big), \,\, \const{K}.
\end{align*}
\end{proposition}

\subsection{Lyapunov functions}

Finally, we recall some facts from \cite{xing2018} concerning Lyapunov functions. 

\begin{definition} \label{def:lyapunov}
Let $f$ and $\sigma$ be as given in \eqref{pde}, and $c$ a constant. A non-negative function $h \in C^2(\R^n)$ is a \textbf{c-Lyapunov function} for $f$ if $h(0) = 0$, $Dh(0) = 0$, and for some $k > 0$ we have
\begin{align*}
    \frac{1}{2} \sum_{i,j = 1}^n (D^2h(y))_{ij} z^i \cdot z^j - Dh(y) \cdot f(t, x, u, \sigma^{-1}(t,x,u) z)  \geq |z|^2 - k
\end{align*}
for all $(t, x ,u, p) \in [0,T] \times \R^d \times \R^n \times (\R^d)^n$ with $|y| \leq c$. In this case, we say that $(h,k) \in \textbf{Ly}(f,c)$.
\end{definition}

The following is a slight adaptation of Proposition 2.11 in \cite{xing2018}.

\begin{lemma} \label{lem.lyapexist}
Suppose that \ref{hyp.quad} and \ref{hyp.sigma} hold. Then for each $c > 0$, there exists a Lyapunov pair $(h,k)$, depending only on $C_{\sigma}$ and $C_Q$, such that $(h,k) \in \text{Ly}(f,c)$ 
\end{lemma}

As a consequence, we get the following.

\begin{lemma} \label{lem.bmoest}
Suppose that \ref{hyp.quad} and \ref{hyp.sigma} hold, and that $u$ is a classical solution to \eqref{pde} with $\norm{u}_{\linf} < \infty$. Then we have 
\begin{align*}
    \sup_{t,x} \norm{Z^{t,x}}_{\bmo} \leq C, \,\, \const{C_{\sigma}, C_Q, \norm{u}_{\linf}}. 
\end{align*}
\end{lemma}

\begin{proof}
For any $(t,x)$, we have
\begin{align*}
    dh(Y^{t,x}_s) &= \bigg(\frac{1}{2} \sum_{i,j} D_{ij} h(Y^{t,x}_s) Z_s^{t,x,i} \cdot Z_s^{t,x,j}  \\
    &\quad - Dh(Y^{t,x}_s) \cdot f(s,X_s^{t,x}, Y^{s,x}, \sigma^{-1}(s,X_s^{t,x},Y_s^{t,x}) Z_s^{t,x} \bigg) ds + dM_s, 
\end{align*}
for some martingale $M$. By using the definition of Lyapunov pair, we get
\begin{align*}
    \bE_{\tau}[\int_{\tau}^T |Z^{t,x}_s|^2 ds] \leq \bE_{\tau}[h(g(X_T^{t,x})) - h(u(\tau,X_{\tau}^{t,x}))+ k(T-\tau)]
    \leq 2\norm{h \circ u}_{\linf} + kT.
\end{align*}
\end{proof}

\section{Proofs of the a-priori estimates for \eqref{pde}} \label{sec.apriori}

Throughout this section, given a decoupling solution $u = (u^i)_{i = 1,..,n} : [0,T] \times \R^d \to \R^n$ to the system \eqref{pde}, and a pair $(t_0,x_0) \in [0,T) \times \R^d$, we will denote by $X^{t_0,x_0}$ the unique strong solution\footnote{We recall that the unique solvability of \eqref{xdef} is part of the definition of a decoupling solution} on $[t_0,T]$ to the stochastic differential equation
\begin{align} \label{xdef}
    X^{t_0,x_0}_t = x_0 + \int_{t_0}^t \sigma(s,X^{t_0,x_0}_s, u(s,X^{t_0,x_0}_s)) dB_s, \quad t \leq s \leq T. 
\end{align}
We will denote by $Y^{t_0,x_0}$ and $Z^{t_0,x_0}$ the processes 
\begin{align} \label{yzdef}
    Y_t^{t_0,x_0} = u(t,X^{t_0,x_0}_t), \quad Z_t^{t_0,x_0} =  \sigma(t,X_t^{t_0,x_0},Y_t^{t_0,x_0})Du(t,X_t^{t_0,x_0}).
\end{align} 
We recall that by hypothesis, the triple $(X^{t_0,x_0}, Y^{t_0,x_0}, Z^{t_0,x_0})$ solves the FBSDE 
\begin{align} \label{fbsde}
    \begin{cases}
    dX^{t_0,x_0}_t = \sigma(t,X^{t_0,x_0}_t,Y^{t_0,x_0}_t) dB_t, \quad t \in [t_0,T] \\
    dY^{t_0,x_0}_t = - f(X^{t_0,x_0}_t,Y^{t_0,x_0}_t,\sigma^{-1}(t,X_t^{t_0,x_0})Z_t^{t_0,x_0}) dt + Z^{t_0,x_0}_t  dB_t \quad t \in [t_0,T], \\
    X^{t_0,x_0}_{t_0} = x_0, \quad Y^{t_0,x_0}_T = g(X_T^{t_0,x_0})
    \end{cases}
\end{align}

We now proceed with the proof of Lemma \ref{lem.ab}.

\begin{proof}[Proof of Lemma \ref{lem.ab}]
Suppose first that \ref{hyp.ab1} holds. For any $(t_0,x_0)$, we set $(X,Y,Z) = (X^{t_0,x_0},Y^{t_0,x_0},Z^{t_0,x_0})$, and notice that 
\begin{align} \label{abest}
    dY_t^i &= - \bigg(Z_t^i \cdot \sigma^{-1}(t,X_t,Y_t) b_0(t,X_t,Y_t,\sigma^{-1}(t,X_t,Y_t)Z_t) + b^i(t,X_t,Y_t, \sigma^{-1}(t,X_t,Y_t)Z_t) \bigg) dt + Z_t^i dB_t \nonumber \\
    &= - b^i(t,X_t,Y_t,\sigma^{-1}(t,X_t,Y_t)Z_t) dt + Z_t d\tilde{B}_t,
\end{align}
where $\tilde{B}$ is a Brownian motion under an equivalent probability measure. We can now apply the reasoning from the proof of Proposition 3.8 in \cite{jackson2021existence} to the pair $(Y, Z)$ to get 
\begin{align*}
    \norm{Y}_{\sinf} \leq C, \, \, \const{\norm{Y_T}_{\linf}, \rho, \sam}, 
\end{align*}
and the result follows. The proof in the case \ref{hyp.ab2} holds is essentially the same, but instead of using the reasoning from Proposition 3.8 in \cite{jackson2021existence} to get from the decomposition \eqref{abest} to the desired estimate, we can instead (because $b^i$ is Lipschitz) use a standard technique for BSDEs with drivers of linear growth, namely studying the dynamics of $\exp{\lambda t} |Y_t|^2$ for large enough $\lambda$. We omit the details.
\end{proof}

\subsection{The H\"older estimate}

This section is devoted to a proof of Theorem \ref{thm.holder}. 

\subsubsection{Preliminaries on Krylov-Safonov estimates and $\bmo$ spaces}
The proof of the H\"older estimate is quite technical and relies on a connection between Krylov-Safonov estimates and $\bmo$-spaces which we learned from \cite{Delarue2003}. This sub-section serves two purposes. The first is to introduce notations and lemmas which will be used in the proof of Theorem \ref{thm.holder}. The second is to demonstrate the connection between the Krylov-Safonov estimates and $\BMO$ martingales in a simpler setting, for the convenience of the reader. As such, we emphasize that while the lemmas and notations in this sub-section are stated precisely, the rest of this sub-section (e.g. the argument for H\"older regularity of the linear PDE \eqref{scalar}) is included to highlight the basic ideas used in the proof of Theorem \ref{thm.holder}, and not meant to be totally rigorous (though it could easily be made so). 

For $(t_0,x_0) \in [0,T) \times \R^d$ and $R \in [0,\sqrt{T-t_0}]$, we define the parabolic cylinder
\begin{align*}
    Q_R(t_0,x_0) = \{(t,x) \in [0,T] \times \R^d : t_0 \leq t \leq t_0 + R^2, \,\, \max_i |x^i - x_0^i| \leq R\}.
\end{align*}
Let us recall a basic fact about functions: in order to prove that a function $v$ is H\"older continuous, it suffices to prove an \textit{decay of oscillation}. In the present parabolic setting, this means that in order to prove that a map $v : [0,T] \times \R^d \to \R$ is locally H\"older continuous on $[0,T)$, it suffices to prove an estimate of the type 
\begin{align} \label{oscdecay}
    \osc_{Q_{R}(t_0,x_0)} v \leq \beta \osc_{Q_{2R}(t_0,x_0)} v + C_0 R^{\gamma}
\end{align}
for some $C_0 > 0$, $\beta \in (0,1)$, $\gamma > 0$, and for all $(t_0,x_0)$, $R$ such that $t_0 + 4R^2 \leq T$. Here, for any subset $U \subset [0,T] \times \R^d$, \begin{align*}
    \osc_U v = \sup_{(t,x) \in U} v(t,x) - \inf_{(t,x) \in U} v(t,x).
\end{align*}
If we want global H\"older estimates, we need to complement the oscillation decay \eqref{oscdecay} with a condition which says that oscillation is small over cylinders which are near the terminal time $T$, i.e. an estimate of the type 
\begin{align} \label{oscdecay2}
    \osc_{Q_{\sqrt{T-t_0}}(t_0,x_0)} v \leq C_0 (T-t)^{\alpha_0/2},  
\end{align}
for some $C_0 > 0$, $\alpha_0 \in (0,1)$ and all $(t,x) \in [0,T) \times \R^d$.

We formalize this discussion with the following Lemma. 

\begin{lemma} \label{lem.suffholder}
Suppose that $v : [0,T] \times \R^d \to \R$ is bounded and satisfies \eqref{oscdecay} and \eqref{oscdecay2}
for some constants $\alpha_0, \beta \in (0,1), \gamma >0$, $C_0 > 0$. Then, for some $\alpha = \alpha(\beta, \gamma, \alpha_0)$, we have 
\begin{align*}
    \norm{v}_{\calpha} \leq C, \,\, \const{C_0, \beta, \gamma, \alpha_0, \norm{v}_{\linf}}
\end{align*}
\end{lemma}

\begin{proof}
In this argument, the constant $C$ may change from line to line and depend on any of the constants $C_0$, $\beta$, $\gamma$, $\alpha_0$ and $\norm{v}_{\linf}$. Let us first record for later use that the estimate \eqref{oscdecay2} implies that the function $g(x) = v(T,x)$ satisfies 
\begin{align*}
    \osc_{B_R(x_0)} g \leq \osc_{Q_{R}(T - R^2, x_0)} v \leq C_0 R^{\alpha_0}, 
\end{align*}
which through a standard argument implies 
\begin{align*}
    \norm{g}_{C^{0,\alpha_0}} \leq C_0.
\end{align*}
Now we fix $(t,x)$ apply Lemma 8.23 of \cite{Gilbarg1977EllipticPD} to the function $\omega(R) = \osc_{Q_R(t,x)} v$, defined on $(0, \sqrt{T-t})$, to conclude that we have 
\begin{align} \label{osc}
    \osc_{Q_R(t,x)} v \leq C R^{\alpha_1} (T -t)^{-\alpha_1/2}  \osc_{Q_{\sqrt{T-t}}(t,x)} v + CR^{\alpha_2} \leq CR^{\alpha_1} (T-t)^{- (\alpha_1 - \alpha_0)/2} + CR^{\alpha_2}, 
\end{align}
for some $\alpha_1, \alpha_2 \in (0,1)$ depending only on $\beta, \gamma, \alpha_0$.
Now, fix $t \in [0,T)$, $x,y \in \R^d$. Suppose first that $\max_i |x^i - y^i| \leq \sqrt{T-t}$. Then setting $R = \max_i |x^i - y^i|$, we have $(t,y) \in Q_{R}(t,x)$
By \eqref{osc}, we conclude
\begin{align} \label{est}
    |v(t,x) - v(t,y)| \leq \osc_{Q_R(t,x)} v &\leq C(T - t)^{-(\alpha_1 - \alpha_0)/2} (\max_i |x^i - y^i|)^{\alpha_1} + C(\max_i |x^i - y^i|)^{\alpha_2}  \nonumber \\
    &\leq C(T-t)^{-(\alpha_1 - \alpha_0)/2} |x- y|^{\alpha_1} + C|x -y|^{\alpha_2}. 
\end{align}
Now if $\alpha_0 \geq \alpha_1$, clearly we have
\begin{align*}
    |v(t,x) - v(t,y)| \leq C |x-y|^{\alpha}, \,\, \alpha = \alpha_0 \wedge \alpha_1 \wedge \alpha_2. 
\end{align*}
If, on the other hand $\alpha_0 < \alpha_1$, then since $\max_i |x^i - y^i| \leq \sqrt{T-t}$, \eqref{est} gives 
\begin{align*}
    |v(t,x) - v(t,y)| &\leq C(T-t)^{-(\alpha_1 - \alpha_0)/2} |x- y|^{\alpha_1 - \alpha_0}|x- y|^{\alpha_0} + C|x - y|^{\alpha_2} \\
    &\leq C|x- y|^{\alpha_0} + C|x - y|^{\alpha_2}. 
\end{align*}
So, at this stage we have established that with $\alpha = \alpha_0 \wedge \alpha_1 \wedge \alpha_2$, we have 
\begin{align} \label{spacereg}
    |v(t,x) - v(t,y)| \leq C|x - y|^{\alpha}
\end{align}
for each $t,x,y$ such that $\max_i |x^i - y^i| \leq \sqrt{T-t}$. If $\max_i |x^i - y^i| > \sqrt{T-t}$, we have \begin{align*}
    |v(t,x) - v(t,y)| &\leq |v(T,x) - v(t,x)| + |v(T,x) - v(T,y)| + |v(T,y) - v(t,y)| \\
    &\leq \osc_{Q_{\sqrt{T-t}}(t,x)} v + \osc_{Q_{\sqrt{T-t}}(t,y)} v + \norm{g}_{C^{0,\alpha_0}} |x-y|^{\alpha_0} \\
    &\leq C(T-t)^{\alpha_0/2} + C|x-y|^{\alpha_0} \leq C|x-y|^{\alpha_0}. 
\end{align*}
So, we have established that \eqref{spacereg} holds for all $t,x,y$. For time regularity, we fix $t,s,x$ with $t \leq s \leq T$. Then since $(s,x) \in Q_{\sqrt{s-t}}(t,x)$, 
\begin{align*}
    |v(t,x) - v(s,x)| \leq \osc_{Q_{\sqrt{s-t}}(t,x)} v \leq C(s-t)^{\alpha_1/2} (T-t)^{-(\alpha_1 - \alpha_0)/2} + C(s-t)^{\alpha_2}. 
\end{align*}
Once again, we split into cases $\alpha_1 > \alpha_0$ and $\alpha_0 \geq \alpha_1$, and in either case we get the estimate 
\begin{align*}
    |v(t,x) - v(s,x)| \leq C(s-t)^{\alpha}. 
\end{align*}
This completes the proof. 
\end{proof}

Now, consider a linear, scalar PDE of the type \begin{align} \label{scalar}
    \partial_t v + \tr(a(t,x)D^2 v) + b(t,x) \cdot Dv + f(t,x)= 0,
\end{align}
with data 
\begin{align*}
    &a(t,x) = \frac{1}{2} \sigma \sigma^T(t,x) :[0,T] \times \R^d \to \R^{d \times d}, \\
    &b = b(t,x) : [0,T] \times \R^d \to \R^d, \quad f = f(t,x) : [0,T] \times \R^d \to \R.
\end{align*}
When $a$, $b$, $f$ are bounded and $a$ is uniformly elliptic, i.e. 
\begin{align*}
    \frac{1}{C_{\sigma}} |z|^2 \leq |\sigma(t,x) z|^2 \leq C_{\sigma} |z|^2. 
\end{align*}
The Krylov-Safonov estimates show that any bounded solution of \eqref{scalar} is locally H\"older continuous on $[0,T) \times \R^d$, with corresponding estimates depending on the $\norm{v}_{\linf}$, and the $\linf$ and ellipticity constants of $b$, $f$, and $\sigma$. Now suppose that $v$ is sufficiently nice and define for $(t_0,x_0) \in [0,T] \times \R^d$ the solution of the SDE \begin{align*}
    X_t^{t_0,x_0} = x_0 + \int_{t_0}^t \sigma(s,X^{t_0,x_0}_s) dB_s. 
\end{align*}
The key to the probabilistic proof of the Krylov-Safonov estimates is the following Lemma, which can be deduced from the results in the original paper \cite{KrySaf79}. We use here and in the remainder of the paper the notation $|A|$ for the Lebesgue measure of a Borel set $A$. 
\begin{lemma} \label{lem.krylov}
Fix $(t_0,x_0)$, $R$ with $t_0 + 4R^2 \leq T$, and let $A \subset Q_{2R}(t_0,x_0)$ with $|A| \geq \frac{1}{2} |Q_{2R}(t_0,x_0)|$. Then for any $(t,x) \in Q_R(t_0,x_0)$, we have 
\begin{align*}
    \bP[\tau_{A} < \tau_{Q_{2R}(t_0,x_0)}] \geq \epsilon, 
\end{align*}
where $\epsilon > 0$ depends only on the $C_{\sigma}$ and 
\begin{align*}
    \tau_A = \inf\{s \geq t : (s,X^{t,x}_s) \in A\}, \quad 
    \tau_{Q_{2R}(t_0,x_0)} = \inf\{s > t : (s,X^{t,x}_s) \in \partial Q_{2R}(t_0,x_0)\}. 
\end{align*}
\end{lemma}
Let us show how by combining Lemma \ref{lem.krylov} with some facts about the space $\bmo$, we can obtain an interior H\"older estimate when $b$ is not necessarily bounded, but satisfies a bound like
\begin{align} \label{bmobound}
    \sup_{(t_0,x_0)} \norm{b(\cdot, X^{t_0,x_0})}_{\bmo} \leq C_0. 
\end{align}
As explained above, we can focus on checking an oscillation estimate like \eqref{oscdecay}. So, we fix $(t_0,x_0) \in [0,T) \times \R^d$, and $R$ such that $t_0 + 4R^2 \leq T$. We also fix $(t,x) \in Q_{R}(t_0,x_0)$, and for simplicity of notation we set $X = X^{t,x}$. Set
\begin{align*}
    &M^+ = \max_{Q_{2R}(t_0,x_0)} v, \quad
    M^- = \min_{Q_{2R}(t_0,x_0)} v, \\
    &A^+ = \{(s,y) \in Q_{2R}(t_0,x_0) : v(s,y) \geq \frac{1}{2}(M^+ + M^-)\}, \\
    &A^- = \{(s,y) \in Q_{2R}(t_0,x_0) : v(s,y) < \frac{1}{2}(M^+ + M^-)\}.
\end{align*}
Obviously, we have one of two alternatives:
\begin{align*}
    |A^+| \geq \frac{1}{2} |Q_{2R}(t_0,x_0)|, \text{ or } \quad |A^-| \geq \frac{1}{2}|Q_{2R}(t_0,x_0)|.
\end{align*}
Let us suppose the second of these two possibilities, the first can be handled by a similar argument. Now, set $\tau = \tau_{A^-} \wedge \tau_{Q_{2R}(t_0,x_0)}$, where
$\tau_{A^-}$ and $\tau_{Q_{2R}}(t_0,x_0)$ are defined as in the proof of Lemma \ref{lem.krylov}.
We know from It\^o's formula that 
\begin{align*}
    dv(t,X_t) &= -\big[ f(t,X_t) - b(t,X_t) \cdot Dv(t,X_t) \big]dt + Dv(t,X_t) \sigma(t,X_t) dB_t \\
    &= - f(s,X_s) + Du(s,X_s) \sigma(s,X_s) d\tilde{B}_s, 
\end{align*}
where $\tilde{B} = B - \int \sigma^{-1}(s,X_s) b(s, X_s) ds$
is a Brownian motion under $\bQ$, with $\frac{d \bQ}{d \bP} = \int  (\sigma^{-1}(s,X_s) b(s, X_s)) dB$. Notice that 
\begin{align*}
    \norm{(\sigma^{-1}(\cdot,X) b(\cdot, X))}_{\bmo} \leq C_{\sigma}  \norm{b(\cdot,X)}_{\bmo} \leq C_{\sigma} C_0. 
\end{align*}
Thus we find that
\begin{align} \label{est2}
    v(t,x) &= \bE^{\bQ}[v(\tau, X_{\tau}) + \int_t^{\tau} f(s,X_s) ds] \nonumber \\
    &\leq \frac{(M^+ + M^-)}{2} \bQ[\tau_{A^-} < \tau_{Q_{2R}(t_0,x_0)}] + M^+ \big(1 - \bQ[\tau_{A^-} < \tau_{Q_{2R}(t_0,x_0)}] \big) + 4R^2 \norm{f}_{\linf}. 
\end{align}
Some arithmetic shows that 
\begin{align*}
    v(t,x) - M^- &\leq (M^+ - M^-) \bigg(1 - \frac{1}{2}  \bQ[ \tau_{A^-} < \tau_{Q_{2R}(t_0,x_0)}] \bigg) + CR^2 \\
    &= \osc_{Q_{2R}(t_0,x_0)} v\bigg(1 - \frac{1}{2}  \bQ[ \tau_{A^-} < \tau_{Q_{2R}(t_0,x_0)}] \bigg) + CR^2.
\end{align*}
Applying Lemma \ref{lem.com} (stated below) and then Lemma \ref{lem.krylov} to estimate from below the quantity $\bQ[ \tau_{A^-} < \tau_{Q_{2R}(t_0,x_0)}]$ lets us conclude that \begin{align*}
    v(t,x) - M^- \leq  \beta \osc_{Q_{2R}(t_0,x_0)} v + CR^2, 
\end{align*}
for some $\beta \in (0,1), C > 0$ depending only on $C_{\sigma}$ and $\norm{f}_{\linf}$. Finally, taking a supremum over $(t,x) \in Q_R(t_0,x_0)$ gives exactly the oscillation decay \eqref{oscdecay}.

\subsubsection{Proof of Theorem \ref{thm.holder}}
Now we give the proof of the H\"older estimate.

\begin{proof}[Proof of Theorem \ref{thm.holder}]
To simplify notation, observe that it suffices to assume that $\sigma = \sigma(t,x)$, but prove an estimate which depends only on the ellipticity constant $C_{\sigma}$ of $\sigma$. That is, we do not assume in this proof that \ref{hyp.sigma} is satisfied, only that $\sigma$ is uniformly elliptic with constant $C_{\sigma}$. So our equation becomes 
\begin{align} \label{semilin}
    \begin{cases} \partial_t u^i + \tr(a(t,x)D^2u^i) + f^i(t,x,u,Du) = 0,\\
    u^i(T,x) = g^i(x), 
    \end{cases}
\end{align}
where $f$ still satisfies \eqref{hyp.quad}, and $g \in \cbeta$. We wish now to prove under these conditions a global H\"older estimate for $u$. 
The idea will be to use the preceding three Lemmas to prove by induction that the following statement holds for each $i$: 
\begin{align} \label{induction}
    &\text{$\norm{u^i}_{\calpha} \leq C$ and }
    \sup_{(t,x)} \sup_{t \leq s - \delta \leq s \leq T} \norm{Z^{t,x,i}1_{[s-\delta,s]}}_{\bmo} \leq C\delta^{\alpha},  \nonumber \\
   &\text{for some constants $C$ and $\alpha$ depending only on $\beta$, $\norm{g}_{\cbeta}$ $C_Q$ $C_{\sigma}$, $\epsilon$, and $\norm{u}_{\linf}$}.
\end{align}
Throughout the argument, constants like $C$, $\alpha$, and $\gamma$ may change freely from line to line but will depend only on $C_{\sigma}, C_Q, \epsilon, \beta, \norm{g}_{\cbeta}$, and $\norm{u}_{\linf}$ unless otherwise stated.
We start with the base case of our induction argument, namely $i = 1$. The idea is to apply (a slightly more sophisticated version of) the argument given above for the linear equation \eqref{scalar} to the equation for $u^1$. For each $(t,x)$, we recall that $X^{t,x}$ denotes the unique solution on $[t,T]$ to \begin{align*}
    X^{t,x}_s = x + \int_t^s \sigma(r,X_r^{t,x}) dB_r, \quad t \leq  s \leq T, 
\end{align*}
and that $(Y^{t,x,1}, Z^{t,x,1}) = (u^1(\cdot,X^{t,x}), \sigma(\cdot, X^{t,x}) Du^1(\cdot, X^{t,x}))$. We begin by establishing the oscillation decay 
\begin{align} \label{oscdecay3}
    \osc_{Q_R(t,x)} u^1 \leq \beta \osc_{Q_{2R}(t,x)} u^1 + CR^{\gamma}, \text{ for all } (t,x) \in [0,T] \times \R^d, \, \, R > 0 \text{ such that } t + 4R^2 \leq T.
\end{align}
for appropriate constants $\beta$, $\gamma$, and $C$. We fix $(t_0,x_0) \in [0,T) \times \R^d$ with $t_0 + 4R^2 \leq T$ and then choose $(t,x) \in Q_R(t_0,x_0)$. Let $h^i$, $k^i$ be as discussed in Remark \ref{rmk.bf} (see in particular \eqref{bfsuff}). Thus we have
\begin{align*}
    \partial_t u^1 + \tr(D^2u^1) + p^1 \cdot h^1(t,x,u,p) + k^1(t,x,u,p), 
\end{align*}
where 
\begin{align} \label{khest}
    |k^1(t,x,u,p)| \leq C_Q(1 + |p|^{2 - \epsilon}), \quad
    |h^1(t,x,u,p)| \leq C_Q(1 + |p|). 
\end{align}
Consequently, we can write 
\begin{align} \label{ydecomp}
    dY^{t,x,1}_s = - \big(Z^{t,x,1}_s \cdot h_s + k_s \big) ds + Z^{t,x,1}_s dB_s
    = - k_s ds + Z^{t,x,1}_s d\tilde{B}_s, 
\end{align}
where we have set 
\begin{align*}
    &h_s = \sigma^{-1}(s,X^{t,x}_s) h^1(s,X_s^{t,x},\sigma^{-1}(s,X^{t,x}_s)Z_s^{t,x}), \quad
     k_s =k^1(s,X_s^{t,x},\sigma^{-1}(s,X^{t,x}_s)Z_s^{t,x}),
\end{align*}
and $\tilde{B} = B - \int h dt$ is a Brownian motion under the measure $d\bQ = \mathcal{E}(\int h dB)$. Notice also that by Lemma \ref{lem.bmoest} and \eqref{khest}  we have $\norm{|k|^{2/(2-\epsilon)}}_{\bmoh} \leq C$, and in particular we have by Lemma \ref{lem.slice}
\begin{align} \label{key}
    \norm{h}_{\bmo} \leq C, \,\, \norm{k_{[s - \delta,s]}}_{\bmoh} \leq C\delta^{\gamma}
\end{align}
for each $t \leq s - \delta \leq s \leq T$ and some $C, \gamma$. We point out for later use that knowing \eqref{key} (for each choice of $(t,x)$) and the bound on $\norm{g^1}_{\cbeta}$ is the only thing that we will use to conclude that \eqref{induction} holds for $i = 1$. 
Now if we set 
\begin{align*}
    &M^+ = \sup_{(s,y) \in Q_{2R}(t_0,x_0)} u^1(s,y), \quad
    M^- = \inf_{(s,y) \in Q_{2R}(t_0,x_0)} u^1(s,y), \\
    &A^+ = \{(s,y) \in Q_{2R}(t_0,x_0) : u^1(s,y) \geq \frac{M^+ + M^-}{2}\}, \\
    &A^- = \{(s,y) \in Q_{2R}(t_0,x_0) : u^1(s,y) \leq \frac{M^+ + M^-}{2}\},
\end{align*}
Then clearly we must either have $|A^+| \geq \frac{1}{2} |Q_{2R}(t_0,x_0)|$ or $|A^-| \geq \frac{1}{2} |Q_{2R}(t_0,x_0)|$. We assume the second possibility, and a symmetric argument will take care of the second. Set $\tau = \tau_{A^-} \wedge \tau_{Q_{2R}(t_0,x_0)}$, and use \eqref{ydecomp} to write 
\begin{align*}
    &u^1(t,x) = \bE^{\bQ}[u^1(\tau, X^{t,x}_{\tau}) + \int_t^{\tau} k_s ds] 
    \\
    &\leq \frac{(M^+ + M^-)}{2}\bQ[\tau_{A^-} \leq \tau_{Q_{2R}(t_0,x_0)}] + M^+ (1- \bQ[\tau_{A^-} \leq \tau_{Q_{2R}(t_0,x_0)}]) + \norm{k}_{\bmoh(\bQ,[t,t_0 + 4R^2])} 
    \\
    &\leq \frac{(M^+ + M^-)}{2}\bQ[\tau_{A^-} \leq \tau_{Q_{2R}(t_0,x_0)}] + M^+ (1- \bQ[\tau_{A^-} \leq \tau_{Q_{2R}(t_0,x_0)}]) + C\norm{k}_{\bmoh([t,t_0 + 4R^2])} 
    \\
    &\leq  \frac{(M^+ + M^-)}{2}\bQ[\tau_{A^-} \leq \tau_{Q_{2R}(t_0,x_0)}] + M^+ (1- \bQ[\tau_{A^-}\leq \tau_{Q_{2R}(t_0,x_0)}]) + C R^{\gamma},
\end{align*}
where the second inequality is given by Lemma \ref{lem.kaz}, and the third is given by \eqref{key} together with Lemma \ref{lem.slice}
Some arithmetic then shows that 
\begin{align*}
    u^1(t,x) - M^- \leq (1 - \frac{1}{2} \bQ[\tau_{A^-} \leq \tau_{Q_{2R}(t_0,x_0)}]) \osc_{Q_{2R}(t_0,x_0)} u^1 + CR^{\gamma}, 
\end{align*}
and since the estimate holds for all $(t,x) \in Q_R(t_0,x_0)$, we conclude that 
\begin{align*}
    \osc_{Q_R(t_0,x_0)} u^1 \leq (1 - \frac{1}{2} \bQ[\tau_{A^-} \leq \tau_{Q_{2R}(t_0,x_0)}]) \osc_{Q_{2R}(t_0,x_0)}u^1 + CR^{\gamma}.
\end{align*}
We now use Lemma \ref{lem.krylov} to bound from below $\bP[\tau_{A^-} \leq \tau_{Q_{2R}(t_0,x_0)}]$, and then \eqref{key} together with Lemma \ref{lem.com} to translate this to an estimate from below on $\bQ[\tau_{A^-} \leq \tau_{Q_{2R}(t_0,x_0)}]$. This allows us to deduce the estimate \eqref{oscdecay}. Next, we establish the estimate
\begin{align} \label{alphanot}
    \osc_{Q_{\sqrt{T-t_0}}(t_0,x_0)} u^1 \leq C (T-t_0)^{\alpha_0/2}
\end{align}
for some constants $\alpha_0$ and $C$. Note first that because $g$ is H\"older continuous, it suffices to show that for some $\gamma$ and $C$ we have
\begin{align} \label{alphanotsuff}
    |u^1(t,x) - g^1(x)| \leq C (T -t)^{\gamma/2}. 
\end{align}
For this, we define $k$, $h$, $\bQ$ as above and notice that
\begin{align*}
    u^1(t,x) = \bE^{\bQ}[g(X_T^{t,x}) + \int_t^T k_s ds], 
\end{align*}
so that 
\begin{align*}
    |u^1(t,x) - g^1(x)| &\leq \bE^{\bQ}[|g^i(X^{t,x}_T) - g^1(x)|] + \bE^{\bQ}[\int_t^T |k_s| ds] \\
    &\leq \bE^{\bQ}[|X^{t,x}_T - x|^{\beta}] + C (T-t)^{\gamma/2} \leq C(T-t)^{\gamma/2},
\end{align*}
where the last inequality follows from the following generalization of Lemma 5.1 in \cite{xing2018}:

\begin{lemma} \label{lem.sliceable2}
Let $\gamma \in \bmo$ and define $\bQ$ by $d\bQ = \mathcal{E}(\int \gamma dB)$. Then for any stopping times $\tau$ taking values in $[s - \delta, s]$ and any $(t,x)$ as above, we have 
\begin{align*}
    \bE^{\bQ}_{\tau}[|X^{t,x}_{s} - X^{t,x}_{\tau}|^{\alpha}] \leq C \delta^{\alpha/2}, 
\end{align*}
where $C$ depends only on $\norm{\sigma}_{\linf}$ and the $\bmo$ norm of $\gamma$. 
\end{lemma}
Postponing the proof of Lemma \ref{lem.sliceable2}, we conclude that \eqref{alphanotsuff} and hence \eqref{alphanot} holds. From Lemma \ref{lem.suffholder}, we can now deduce that $\norm{u^1}_{\calpha} \leq C$.
Now we show how this implies the second part of the statement \eqref{induction}, namely the estimate 
\begin{align*}
    \sup_{(t,x)} \sup_{t \leq s - \delta \leq s \leq T} \norm{Z^{t,x,1}1_{[s-h,s]}}_{\bmo} \leq C\delta^{\alpha}. 
\end{align*}
We fix $(t,x)$ and define $X^{t,x}, Y^{t,x}, Z^{t,x}, h$, $q$, etc. as above and compute 
\begin{align*}
    d|Y^{t,x,1}_u|^2 = (-2Y^{t,x,i}_u k_s + |Z_u^{t,x,1}|^2) ds + dM_s, 
\end{align*}
where $M$ is a martingale under $\bQ$. Thus given a stopping time $\tau$ with $s - \delta \leq \tau \leq s$, we have
\begin{align*}
    \bE^{\bQ}_{\tau}[\int_{\tau}^s |Z^{t,x,1}_u|^2 du] &= \bE^{\bQ}_{\tau}[ |Y^{t,x,1}_s|^2 - |Y^{t,x,1}_{\tau}|^2 + 2\int_{\tau}^s Y^{t,x,1}_u k_u du] \\
    &\leq \bE^{\bQ}_{\tau}[ |u^1(s,X^{t,x}_s)|^2 - |u^1(\tau,X^{t,x}_{\tau})|^2 + C\norm{k}_{\bmoh(\bQ,[s- \delta,s])} \\
    &\leq
    \bE^{\bQ}_{\tau}[ |s - \tau|^{\alpha/2} + |X^{t,x}_s - X^{t,x}_{\tau}|^{\alpha} ] + C\norm{k}_{\bmoh([s- \delta,s])} \\ 
    &\leq C\delta^{\alpha} + \bE_{\tau}^{\bQ}[|X_s^{t,x} - X_{\tau}^{t,x}|^{\alpha}] \leq C\delta^{\alpha},
\end{align*}
where we have once again used \eqref{key} and Lemma \ref{lem.kaz}, and the last line follows from Lemma \ref{lem.sliceable2}. We have now shown that $\bE^{\bQ}_{\tau}[\int_{\tau}^s |Z^{t,x,1}_u|^2 du] \leq C \delta^{\alpha}$ whenever $t \leq s - \delta \leq \tau \leq s$, 
from which it follows that
\begin{align*}
    \norm{Z^{t,x,1}1_{[s-\delta,s]}}_{\bmo} \leq C \norm{Z^{t,x,1}1_{[s-\delta,s]}}_{\bmo(\bQ)} \leq C\delta^{\alpha}. 
\end{align*}
Thus we have established \eqref{induction} in the case $i = 1$. The induction step is almost exactly the same. Suppose we know that \eqref{induction} holds for all $i < j$. Then we may again use the decomposition in Remark \ref{rmk.bf} to write
\begin{align*}
    \partial_t u^i + \tr(a D^2 u^i) + k^i(t,x,u,p) + p^i \cdot h^i(t,x,u,p)=0, 
\end{align*}
where
\begin{align*}
    |k^i(t,x,u,p)| \leq C_Q(1 + |p|^{2- \epsilon} + \sum_{j < i} |p|^j), \quad |h^i(t,x,u,p)| \leq C_Q(1 + |p|). 
\end{align*}
Using the induction hypothesis and Lemma \ref{lem.slice}, we conclude that for any $(t,x)$, if we define
\begin{align*}
    k_s = k^i(s, X^{t,x}_s, Y^{t,x}_s,\sigma^{-1}(s,X^{t,x}_s) Z^{t,x}_s), \quad h_s = h^i(s, X^{t,x}_s, Y^{t,x}_s,\sigma^{-1}(s,X^{t,x}_s) Z^{t,x}_s), 
\end{align*}
then we have 
\begin{align} \label{keyi}
    \norm{h}_{\bmo} \leq C, \quad \norm{k_{[s-\delta,s]}}_{\bmo} \leq C\delta^{\alpha}. 
\end{align}
As in the case $i = 1$, \eqref{keyi} together with the control of $\norm{g^i}_{\calpha}$ is enough to conclude that \eqref{induction} holds for $i$. This completes the proof. 

\end{proof}

\begin{proof}[Proof of Lemma \ref{lem.sliceable2}]
For simplicity, we write $X = X^{t,x}$. Then we have 
\begin{align*}
    dX_u = \sigma(u,X_u) dB_u = \sigma(u,X_u) d\tilde{B}_u + \sigma(u,X_u)\gamma_u du, 
\end{align*}
so with $s$ and $\tau$ as in the statement of the Lemma,
\begin{align*}
    \bE^{\bQ}_{\tau}[|X_s - X_{\tau}|] \leq &\bE^{\bQ}_{\tau}[\int_{\tau}^s \sigma(u,X_u) d\tilde{B}_u] + \bE_{\tau}^{\bQ}[\int_{\tau}^s \sigma(u,X_u) \gamma_u du] \\
    &\leq C\delta^{1/2} + C \norm{\gamma1_{[s-\delta,s]}}_{\bmoh(\bQ)} \leq C\delta^{1/2} + C \norm{\gamma1_{[s-\delta ,s]}}_{\bmoh} \\
     &\leq C\delta^{1/2} + C \norm{\gamma 1_{[s-\delta ,s]}}_{\bmoh} 
    \leq C\delta^{1/2}, 
\end{align*}
where the third inequality uses Lemma \ref{lem.kaz}. 
To complete the proof, note that for $\alpha \in (0,1)$, 
\begin{align*}
    \bE^{\bQ}_{\tau}[|X_s - X_{\tau}|^{\alpha}] \leq \bE^{\bQ}_{\tau}[|X_s - X_{\tau}|]^{\alpha} \leq C \delta^{\alpha/2}.
\end{align*}
\end{proof}



\subsection{The gradient bound}

This section is devoted to a proof of Theorem \ref{thm:gradest}. We begin with a Lemma which explains that H\"older estimates always lead to estimates on the sliceability of the processes $Z^{t,x}$, provided that we have a Lyapunov function.
\begin{lemma} \label{lem.sliceable}
Suppose that $u$ is a decoupling solution to \eqref{pde}, and $\norm{u}_{\linf} \leq c$. Suppose further that there exits $(h,k) \in \mathbf{Ly}(f,c)$. Finally, suppose that we have $\norm{u}_{\calpha} < \infty$, for some $\alpha \in (0,1)$. Then, for any $(t_0,x_0) \in [0,T) \times \R^d$ and any stopping time $\tau$ and $t \in [0,T]$ with $(t_0 \vee (t - \delta)) \leq \tau \leq t$, we have 
\begin{align*}
   \bE_{\tau}[ \int_{\tau}^t |Z^{t_0,x_0}_s|^2 ds] \leq k\delta + C\norm{Dh}_{\linf(B_{\norm{u}_{\linf}})}\norm{u}_{\calpha} (t-s)^{\delta/2}, 
\end{align*}
where $C = C(\norm{\sigma}_{\linf})$. In particular, $Z^{t_0,x_0}$ is sliceable, with an index of sliceability independent of $(t_0,x_0)$, and depending only on $h$, $k$, $\norm{u}_{\calpha}$, and $\norm{\sigma}_{\linf}$. 
\end{lemma}




\begin{proof}
The proof is essentially the same as that of Proposition 5.2 in \cite{xing2018}. Namely, fixing any $(t_0,x_0)$ and setting $(X,Y,Z) = (X^{t_0,x_0}, Y^{t_0,x_0}, Z^{t_0,x_0})$ for simplicity, we have
\begin{align*}
    dh(Y^{t_0,x_0}_s) = \big(\frac{1}{2} \sum_{i,j} D_{ij} h(Y_s) Z_s^i \cdot Z_s^j  - Dh(Y_s) \cdot f(X_s,Y_s,\sigma^{-1}(s,X_s) Z_s) \big) ds + dM_s, 
\end{align*}
for some martingale $M$. Applying the definition of a Lyapunov pair, we find that 
\begin{align*}
    \bE_{\tau}[ h(Y_{t}) - h(Y_{\tau}) ] \geq \bE_{\tau}[\int_{\tau}^t |Z_s|^2 ds] - k\delta. 
\end{align*}
We conclude the proof by estimating
\begin{align*}
    \bE_{\tau}[h(Y_t) - h(Y_{\tau})] &\leq \norm{Dh}_{\linf(B_{\norm{u}_{\linf}})}\norm{u}_{\calpha} \big(\delta^{\alpha/2} + \bE_{\tau}[|X_t - X_{\tau}|^{\alpha}] \big) \\ &\leq C \norm{Dh}_{\linf(B_{\norm{u}_{\linf}})}\norm{u}_{\calpha} \delta^{\alpha/2}, 
\end{align*}
where the last inequality comes from Lemma 5.1 of \cite{xing2018}. 
\end{proof}

Now we present the proof of the gradient bound, Theorem \ref{thm:gradest}. 
\begin{proof}[Proof of Theorem \ref{thm:gradest}]
For notational simplicity, we give the argument in the case $d = 1$, but the same argument goes through in when $d > 1$. In this case, our equation becomes 
\begin{align} \label{1d}
    \partial_t u^i + a(t,x,u) D^2 u^i + f^i(t,x,u,Du) = 0, \quad (t,x) \in [0,T) \times \R, 
\end{align}
with the terminal condition $u^i(T,x) = g^i(x)$. We now compute the equations for $v^i = Du^i$, and find 
\begin{align*}
    \partial_t v^i + a(t,x,u) D^2 v^i + (D_x a(t,x,u) + D_u a(t,x,u) \cdot Du) Dv^i + D_x f^i(t,x,u,Du) \\ + D_u f^i(t,x,u,Du) \cdot v + D_p f^i(t,x,u,Du) \cdot Dv = 0, 
\end{align*}
with the terminal condition $v^i(T,x) = Dg^i(x)$. We fix $(t_0,x_0)$ and set $X = X^{(t_0,x_0)}$, where we continue to define $X^{(t_0,x_0)}$ by \eqref{xdef}. We then set
\begin{align*}
    Y = v(\cdot, X) = Du(\cdot, X), \quad Z = \sigma(\cdot, X) Dv(\cdot, X) = \sigma(\cdot, X) D^2 u(\cdot, X), \quad
    U = u(\cdot, X),
\end{align*}
so that $(Y,Z)$ solves the linear BSDE 
\begin{align*}
    Y_t^i
    =  \xi^i + \int_t^T \big(\alpha_s^i \cdot Y_s + A_s^i \cdot Z_s + \beta_s^i \big) ds - \int_t^T Z_s^i dB_s, 
\end{align*}
where 
\begin{align*}
&\xi = Dg(X_T), \quad
\alpha_t^i = D_u f^i(t,X_t,U_t, Y_t), \\
&A_t^i = \sigma^{-1}(t,X_t,U_t) \Big[ D_p f^i(t,X_t, U_t,Y_t) + (D_x a(t,X_t,U_t) + D_u a(t,X_t,U_t) \cdot Y_t)e_i \Big], \\
&\beta_t^i = D_x f^i(t,X_t,U_t, Y_t). 
\end{align*}
Here we use $e_i$ to denote the $i^{th}$ standard basis vector of $\R^n$.
Now because of Theorem \ref{thm.holder}, Lemma \ref{lem.sliceable}, and Lemma \ref{lem.lyapexist}, we can find a $K : (0,\infty) \to \N$ depending on $C_{\sigma}, L_{\sigma}, C_Q, \alpha, \norm{u}_{\calpha}, \epsilon$ such that 
\begin{align*}
    N_{\alpha}(\delta) + N_{A}(\delta) \leq K(\delta)
\end{align*}
for each $\delta > 0$. Moreover, we have $\norm{\beta}_{\bmoh} \leq C, \,\, \const{C_Q, \norm{Y}_{\bmo}}$. The result now follows from Proposition \ref{prop:sliceableeqns}.

\end{proof}

\section{Proofs of the a-priori estimates for \eqref{pde2}} \label{sec.apriori2}

In this section, given a classical solution $u = (u^i)_{i = 1}^n : [0,T] \times \R^d \to \R^n$ to the system \eqref{pde2}, and a pair $(t_0,x_0) \in [0,T) \times \R^d$, we will denote by $X^{t_0,x_0}$ the unique strong solution on $[t_0,T]$ to the stochastic differential equation
\begin{align} \label{xdef2}
    X^{t_0,x_0}_t = x_0 + \int_{t_0}^t \sigma(s,X_s,u(s,X_s), Du(s,X_s)) dB_s, \quad t \leq s \leq T. 
\end{align}
We will denote by $Y^{t_0,x_0}$ and $Z^{t_0,x_0}$ the processes 
\begin{align} \label{yzdef2}
    Y_t^{t_0,x_0} = u(t,X^{t_0,x_0}_t), \quad Z_t^{t_0,x_0} = \sigma(t,X_t^{t_0,x_0},Y_t^{t_0,x_0}, Du(t,X^{t_0,x_0}))Du(t,X^{t_0,x_0}).
\end{align} 
We note that $(Y^{t_0,x_0},Z^{t_0,x_0})$ satisfies
\begin{align} \label{fbsde2}
    \begin{cases}
    dY^{t_0,x_0}_t = - f(X^{t_0,x_0}_t,Y^{t_0,x_0}_t,\sigma^{-1}(t,X^{t_0,x_0},Y_t^{t_0,x_0}, Du(t,X_t^{t_0,x_0}))Z_t^{t_0,x_0}) dt + Z^{t_0,x_0}_t  dB_t \quad t \in [t_0,T], \\
    Y^{t_0,x_0}_T = g(X_T^{t_0,x_0})
    \end{cases}
\end{align}

\begin{proof}[Proof of Theorem \ref{thm.apriori}]
In this proof, the constant $C$ will change line to line but only depend on the quantities listed in the statement of Theorem \ref{thm.apriori}. First, we note that by using the probabilistic representation \eqref{fbsde2} together with \eqref{hyp.lip}, it is straightforward to get an estimate of the form 
\begin{align} \label{linfest.}
    \norm{u}_{\linf} \leq C.
\end{align}
In fact, once we have this $\linf$ estimate, we can also use Theorem \ref{thm.holder} to obtain an estimate of the form 
\begin{align} \label{calphaest.}
    \norm{u}_{\calpha} \leq C, 
\end{align}
with $\alpha$ depending only on $C_{\sigma}$, $\norm{u}_{\linf}$, and $C_f$. Indeed, notice that under the hypotheses of Theorem \ref{thm.apriori}, the function $\tilde{\sigma}(t,x) = \sigma(t,x,u(t,x), Du(t,x))$ satisfies \eqref{hyp.sigma}, with the same ellipticity constant $C_{\sigma}$ as in \eqref{hyp.sigma2}. The Lipschitz constant of $\tilde{\sigma}$ will depend, of course, on regularity of $u$, but since the estimate in Theorem \ref{thm.holder} does not depend on $L_{\sigma}$, we can infer \eqref{calphaest.} from Theorem \ref{thm.holder} and \eqref{linfest.}.

Next, we differentiate the equation \eqref{pde2}. Setting $v^i = Du^i$, we find 
\begin{align*}
    \partial_t v^i + a(t,x,u,v) D^2 v^i + \bigg(D_x a(t,x,u,v) + D_u a(t,x,u,v) \cdot v + D_p a(t,x,u,v) \cdot Dv\bigg) Dv^i \\
    + D_x f^i(t,x,u,v) + D_u f^i(t,x,u,v) \cdot v + D_p f^i(t,x,u,v)  \cdot Dv = 0,
\end{align*}
with the terminal condition $v^i(T,x) = Dg^i(x)$. We can rewrite this as 
\begin{align*}
    \partial_t v^i + \tilde{A}(t,x) D^2 v^i + \tilde{f}(t,x,u,Du) = 0,  
\end{align*}
where 
\begin{align*}
    &\tilde{A}(t,x) = \frac{1}{2} \tilde{\sigma} \tilde{\sigma}^T(t,x), \quad
    \tilde{\sigma}(t,x) = \sigma(t,x,u(t,x),Du(t,x))
\end{align*}
and 
\begin{align*}
    \tilde{f}^i(t,x,v,p) &= \bigg(D_x a(t,x,u(t,x),Du(t,x)) + D_u a(t,x,u(t,x),Du(t,x)) \cdot v  \\ &\qquad + D_p a(t,x,u(t,x), Du(t,x)) \cdot p \bigg) p^i \\ &\qquad + D_x f^i(t,x,u(t,x),Du(t,x)) + D_uf^i(t,x,u(t,x),Du(t,x)) \cdot v \\ &\qquad+ D_p f^i(t,x,u(t,x), Du(t,x)) \cdot p
\end{align*}
Using \eqref{hyp.sigma2}, \eqref{hyp.lip}, and the bound already established on $\norm{u}_{\linf}$, we can check that the data $\tilde{A}$, $\tilde{f}$, satisfy the hypotheses \ref{hyp.ab2}, \ref{hyp.quad}, and \ref{hyp.sigma} (and with the relevant constants depending only on $C_{\sigma}$, $L_{\sigma}$, $C_f$). Moreover, clearly the terminal condition $Dg \in C^{\beta}$. We can thus apply Theorem \ref{thm.holder} to complete the proof of the estimate on $\norm{Du}_{\calpha}$ (with a smaller $\alpha$ if necessary). 

Now suppose in addition we have $g \in C^{2,\beta}$ and \ref{hyp.reg2} holds. To get the estimate on $C^{2, \alpha}$, we would like to appeal to Schauder theory. 
We write the equation for $u^i$ as 
\begin{align} \label{frozen}
    \partial_t u^i + \tilde{A}(t,x)D^2 u^i + \tilde{f}^i(t,x) = 0, 
\end{align}
where 
\begin{align*}
    \tilde{A}(t,x) = a(t,x,u(t,x),Du(t,x)), \quad \tilde{f}^i(t,x) = f^i(t,x,u(t,x), Du(t,x)). 
\end{align*}
Using the estimates so far obtained on $u$ and $Du$, we can check that $\tilde{A}(t,x) \in \calpha$, $\tilde{f}^i \in \calpha$ (again, updating $\alpha$ if necessary, and with corresponding quantitative estimates). Now we can appeal to the classical Schauder estimates to conclude the desired estimate on $\norm{u}_{C^{2,\alpha}}$.
\end{proof}

\section{Proof of the existence results}
\label{sec.existence}
\begin{proof}[Proof of Theorem \ref{thm:existence}]
As explained in Remark \ref{rmk.unique}, we need only prove existence. The idea is to first truncate and then mollify the data, and then pass to the limit using a compactness argument. For each $k$, we define $\pi^{(k)} : (\R^d)^n \to (\R^d)^n$ by 
\begin{align*}
    \pi^{(k)}(p) = \begin{cases}
    p & |p| \leq k, \\
    \frac{k p}{|p|} & |p| > k.
    \end{cases}
\end{align*}
We define for each $k \in \N$ a driver $f^{(k)}$ by 
\begin{align} \label{fk}
    f^{(k),i}(t,x,u,p) = f^{i}(t,x,y,\pi^{(k)}(p)). 
\end{align}
Next, we let $(\rho_{\epsilon})_{0 < \epsilon < 1}$ be a standard mollifier on $\R \times \R^d \times \R^n \times (\R^d)^n$, and we set \begin{align*}
   f^{(k),\epsilon,i}(t,x,u,p) = \int_{\R \times \R^d \times \R^n \times (\R^d)^n} f^{(k),i}(t',x',u',p') \rho_{\epsilon}(t' - t,x-x', u' -u, p'-p) dt' dx' du' dp',  
\end{align*}
where $f^{(k),i}(t,x,u,p)$ has been extended to $t \in \R$ by $f^{(k),i}(t,x,u,p) = f^{(k),i}((t \vee 0) \wedge T, x,u,p)$.
Likewise, define $\sigma^{\epsilon}$ through a standard mollification in $(t,x)$. Since $f^{(k),\epsilon}$ is bounded with bounded derivatives of all orders, it is standard that for each $k \in \N$ and $0 < \epsilon < 1$, there is a unique classical solution $u^{(k),\epsilon}$ to the equation 
\begin{align} \label{pdemoll}
    \begin{cases}
    \partial_t u^{(k),\epsilon,i} +  \tr(\sigma^{\epsilon} (\sigma^{\epsilon})^T(t,x,u^{(k),\epsilon}) D^2u^{(k),\epsilon,i}) + f^{(k),\epsilon,i}(t,x,u^{(k),\epsilon},Du^{(k),\epsilon}) = 0, \quad (t,x) \in (0,T) \times \R^d, \\
    u^{(k),\epsilon,i}(T,x) = g^i(x), \quad x \in \R^d,
    \end{cases}
\end{align}
see e.g. Proposition 3.3 of \cite{ma1994}. Some tedious but straightforward computations verify that the data $f^{(k),\epsilon}$, $\sigma^{\epsilon}$ satisfy the hypotheses \ref{hyp.ab} and \ref{hyp.reg} uniformly in $k$ and $\epsilon$. In particular, by Lemma \ref{lem.ab} and Theorems \ref{thm.holder} and \ref{thm:gradest}, we may conclude that 
\begin{align} \label{abest.}
    ]\sup_{k,\epsilon} \norm{u^{(k),\epsilon}}_{\linf} + \sup_{k,\epsilon} \norm{Du^{(k),\epsilon}}_{\linf} < \infty.
\end{align}
But now for a smooth cut-off function $\kappa : \R^d \to \R$ with $\kappa(x) = 1$ for $|x| \leq 1$, $\kappa(x) = 0$ for $|x| \geq 2$, we can compute for each $x_0 \in \R^d$ the equation satisfied by $u^{(k),\epsilon, x_0} = u^{(k),\epsilon}\kappa(x - x_0)$, and we find that that each component of $u^{(k), \epsilon, x_0}$ satisfies a linear parabolic equation with uniformly H\"older continuous coefficients with a right-hand side which is bounded uniformly in $x_0, k,\epsilon$. Then applying the Calderon-Zygmund estimates for this equation (see e.g. Theorem 1 in Chapter 5, Section 2 of \cite{krylovsob}) gives 
\begin{align} \label{krylovest}
    \sup_{k,\epsilon} \sup_{x_0 \in \R^d} \int_0^T \int_{B_1(x_0)} \big(|\partial_t u^{(k),\epsilon}|^p + |D u^{(k),\epsilon}|^p + |D^2u^{(k),\epsilon}|^p \big) dx dt < \infty, 
\end{align}
for each $p < \infty$.\footnote{To be precise, because of the term involving $\lambda$ appearing in the statement of the cited result in \cite{krylovsob}, we need to use the fact that $\sup_{k,\epsilon} \|u^{(k),\epsilon}\|_{\linf} < \infty$ (which has been noted already in \eqref{abest.}) in order to apply the result of \cite{krylovsob} and obtain \eqref{krylovest}.}
By Sobolev embedding (see Appendix E of \cite{fleming} for a nice review of parabolic Sobolev embedding and \cite{lady} for the proofs), we conclude that for each $0 < \gamma < 1$, we have
\begin{align*}
    \sup_{k, \epsilon} \norm{u^{(k),\epsilon}}_{C^{1,\gamma}} < \infty. 
\end{align*}
Now with 
\begin{align*}
\tilde{f}^{(k),\epsilon,i}(t,x) = f^{(k),\epsilon,i}(t,x,u^{(k),\epsilon}, Du^{(k),\epsilon}), 
\end{align*}
we deduce that for some $\gamma \in (0,1)$, 
\begin{align*}
    \sup_{k,\epsilon} \norm{\tilde{f}^{(k),\epsilon,i}(t,x)}_{C^{0,\gamma}} < \infty, 
\end{align*}
and so viewing \eqref{pdemoll} as a linear equation and applying the Schauder estimates (see Theorem 9.2.2 in \cite{krylovholder}), we get 
\begin{align*}
    \sup_{k,\epsilon} \norm{u}_{C^{2,\gamma}} < \infty. 
\end{align*}
This allows us to find $u \in C^{2,\gamma}$, $k_j \uparrow \infty$, $\epsilon_j \downarrow 0$ such that 
\begin{align*}
    u^{(k_j), \epsilon_j} \to u, \quad 
    \partial_t u^{(k_j), \epsilon_j} \to \partial_t u, \\
    Du^{(k_j), \epsilon_j} \to Du, \quad
    D^2u^{(k_j), \epsilon_j} \to Du
\end{align*}
locally uniformly on $[0,T] \times \R^d$. Then it is clear that $u$ is the desired classical solution to \eqref{pde}. 
\end{proof}

\begin{proof}[Proof of Theorem \ref{thm:existencedecoup}]
As pointed out in Remark \ref{rmk.unique}, we need only prove existence. The proof is very similar to that of Theorem \ref{thm:existence}, so we are brief here. Let $\sigma^{\epsilon}$, $f^{(k),\epsilon}$ be defined exactly as in the proof of Theorem \ref{thm:existence}, and let $g^{\epsilon}$ be a standard mollification of $g$. Then let $u^{(k),\epsilon}$ be the unique classical solution $u^{(k),\epsilon}$ to 
\begin{align} 
    \begin{cases}
    \partial_t u^{(k),\epsilon,i} +  \tr(\sigma^{\epsilon} (\sigma^{\epsilon})^T(t,x,u^{(k),\epsilon}) D^2u^{(k),\epsilon,i}) + f^{(k),\epsilon,i}(t,x,u^{(k),\epsilon},Du^{(k),\epsilon}) = 0, \quad (t,x) \in (0,T) \times \R^d, \\
    u^{(k),\epsilon,i}(T,x) = g^{\epsilon,i}(x), \quad x \in \R^d. 
    \end{cases}
\end{align}
As in the proof of Theorem \ref{thm:existence}, we know that for some $\alpha \in (0,1)$,
\begin{align*}
    \sup_{k,\epsilon} \norm{Du^{(k),\epsilon}}_{\linf} < \infty, \quad
    \sup_{k,\epsilon} \norm{u^{(k),\epsilon}}_{\calpha} < \infty
\end{align*}
This time, there is no way to bootstrap to conclude a uniform bound in $C^{2,\alpha}$. Instead, we can fix a smooth cutoff function $\rho = \rho(x) : \R^d \to \R$ with $0 \leq \rho \leq 1$ and $\rho(x) = 1$ for $|x| \leq 1$, $\rho(x) = 0$ for $|x| \geq 2$ and a smooth cutoff function $\kappa = \kappa(t) : [0,T] \times \R$ with $0 \leq \kappa \leq 1$, $\kappa = 1$ for $t \leq T -\delta$, $\kappa = 0$ for $t > T - \delta/2$. Then for any $x_0$, $k, \epsilon$ and $i$ the function $\tilde{u}^{(k),\epsilon,i}(t,x) = u^{k,\epsilon,i}(t,x) \rho(x - x_0) \kappa(t)$ satisfies a linear parabolic of the form \begin{align*}
    \partial_t \tilde{u}^{(k),\epsilon,i} + \tr(\tilde{a}^{(k),\epsilon} D^2 \tilde{u}^{(k),\epsilon,i}) + \tilde{f}^{(k),\epsilon,i} = 0, \quad \tilde{u}^{(k),\epsilon,i}(T,x) = 0, 
\end{align*}
with $\tilde{a}^{(k),\epsilon}$ elliptic uniformly in $k,\epsilon$ and the estimates
\begin{align*}
    |\tilde{a}^{(k),\epsilon}(t,x) - \tilde{a}^{(k),\epsilon}(t,x')| \leq C|x - x'|, \quad \norm{f^{(k),\epsilon,i}}_{\linf} \leq C
\end{align*}
holding for all $k$, $\epsilon$, $i$, with a constant $C$ independent of $x_0$. By applying Theorem 1 of Chapter 5, Section 2 of \cite{krylovsob}, we get for any fixed $\delta > 0$ the estimate
\begin{align*}
    \sup_{k,\epsilon} \sup_{x_0 \in \R^d} \int_0^{T-\delta} \int_{B_1(x_0)} \big(|\partial_t u^{(k),\epsilon}|^p + |D u^{(k),\epsilon}|^p + |D^2u^{(k),\epsilon}|^p \big) dx dt < \infty, 
\end{align*}
hence by Sobolev embedding
\begin{align} \label{unif}
    \sup_{k,\epsilon} \big(\norm{Du^{(k),\epsilon}}_{\linf} + \norm{u^{(k),\epsilon}}_{\calpha}  + \norm{u}^{(k),\epsilon}_{C^{1,\alpha}([0,T-\delta] \times \R^d)}\big) < \infty
\end{align}
for each $\delta > 0$. This lets us find a function 
\begin{align*}
    u \in \calpha([0,T] \times \R^d) \cap C^{1, \alpha}_{\text{loc}}([0,T) \times \R^d)
\end{align*}
and sequences $k_j \uparrow \infty$, $\epsilon_j \downarrow 0$ such that
\begin{align} \label{conv}
    &u^{(k_j),\epsilon_j} \to u \text{ locally uniformly in } [0,T] \times \R^d, \nonumber \\
    &Du^{(k_j),\epsilon_j} \to Du \text{ locally uniformly in } [0,T) \times \R^d,
\end{align}
and $\norm{Du}_{\linf} < \infty$. For simplicity, let us set 
\begin{align*}
    u^{(j)} = u^{(k_j),\epsilon_j}, \quad f^{(j)} = f^{(k_j), \epsilon_j}, \quad \sigma^{(j)} = \sigma^{\epsilon_j}, \quad g^{(j)} = g^{\epsilon_j}
\end{align*}
Now fix $(t_0,x_0)$, and define processes $(X^{(j)} Y^{(j)}, Z^{(j)})$ and $(X,Y,Z)$ by
\begin{align*}
    &X_t^{(j)} = x_0 + \int_{t_0}^t \sigma^{(j)}(s,X^{(j)}_s, u^{(j)}(s,X^{(j)})) dB_s, \quad t_0 \leq t \leq T \\
    &X_t = x_0 + \int_{t_0}^t \sigma(s,X_s, u(s,X)) dB_s \quad t_0 \leq t \leq T
    \end{align*}
    and 
    \begin{align*}
    &Y_t^{(j)} = u^{(k_j)}(t, X^{(j)}_t), \quad Y_t = u(t, X_t), \\
    &Z_t^{(j)} = \sigma^{(j)}(t,X_t^{(j)})Du^{(j)}(t,X_t^{(j)}), \quad Z_t = \sigma(t,X_t)Du(t,X_t).
\end{align*}
For each $j$, It\^o's formula gives us the relation
\begin{align} \label{fbsdecompact}
    Y^{(j)}_t = g^{\epsilon_j}(X_T^{(j)}) + \int_t^T  f^{(j)}(s,X_s^{(j)}, Y^{(j)}_s, Z^{(j)}_s) ds - \int_{t}^T Z^{(j)}_s dB_s. 
\end{align}
The fact that $u$ is Lipschitz and $u^{(j)} \to u$ uniformly is enough to conclude that $X^{(j)} \to X$ in $\stwo$, and then \eqref{conv} is enough to conclude that $Y^{(j)} \to Y$ in $\stwo$, $Z^{(j)} \to Z$ in $\ltwo$. This is enough to pass to the limit in \eqref{fbsdecompact} and conclude that 
\begin{align}
    Y_t = g(X_T) + \int_t^T  f(s,X_s, Y_s, Z_s) ds - \int_{t}^T Z_s dB_s, 
\end{align}
which means that $u$ is a decoupling solution for \eqref{pde}.

\end{proof}

\begin{proof}[Proof of Theorem \ref{thm.exist2}]
First, suppose that $f$, $\sigma$ and $g$ are smooth with bounded derivatives of all orders. In this casee, we will establish existence via the method of continuation, and follow closely the presentation in Chapter 17 of \cite{Gilbarg1977EllipticPD}. Since $g$ is smooth, we may as well assume that $g = 0$ (otherwise we can study the system satisfied by $\tilde{u}^i(t,x) = u^i(t,x) - g(x)$). Let $\alpha \in (0,1)$ be as given by Theorem \ref{thm.apriori}, and define the Banach spaces $B_1$ and $B_2$ by 
\begin{align*}
    B_1 = \{u \in C^{2, \alpha}([0,T] \times \R ; \R^n) : u(T,x) = 0\}, \quad
    B_2 = \calpha([0,T] \times \R ;\R^n). 
\end{align*}
Now fix an arbitrary $\phi \in B_1$, and define the functional $F = F(u,\lambda) : B_1 \times [0,1] \to B_2$ by
\begin{align*}
    F^i(u,\lambda) = \partial_t u^i + \big(\lambda a(t,x,u,Du) + (1 - \lambda)\big) D^2 u^i + f^i(t,x,u,Du). 
\end{align*}
Define $\Lambda \subset [0,1]$ by 
\begin{align*}
    \Lambda = \{\lambda \in [0,1] : F(u,\lambda) = 0 \text{ for some } u \in B_1\}. 
\end{align*}
Theorem \ref{thm:existence} shows that $0 \in \Lambda$. We next claim that the a-priori estimate Theorem \ref{thm.apriori} implies that $\Lambda$ is closed. Indeed, if for $k \in \N$ we have $F(u^k,\lambda^k) =0$ and $\lambda^k \to \lambda \in [0,T]$, then Theorem \ref{thm.apriori} implies that $\{u^k\}$ is compact in $C^{2, \beta}$ for any $\beta < \alpha$, and this lets us find a $u \in B_1$ such that (up to a subsequence) $u^k \to u$ in $C_{\text{loc}}^{2, \beta}([0,T] \times \R^d ; \R^n)$, and so $F(u,\lambda) = 0$, and $\lambda \in \Lambda$. 

To see that $\Lambda$ is open, notice that the Frechet derivative $D_u F$ of $F$ in the first argument is given by 
\begin{align*}
    \big(D_u F(u,\lambda)(v)\big)^i = \partial_t v^i + \lambda \tr(a(t,x,u,Du) D^2 v^i) + (1 - \lambda) \Delta v^i + \lambda \big(D_u a(t,x,u,Du) \cdot v \\ + D_p a(t,x,u,Du) \cdot Dv \big) D^2 u^i + D_u f^i(t,x,u,Du) \cdot v + D_p f^i(t,x,u,Du) \cdot Dv. 
\end{align*}
It follows from results on solvability of linear parabolic systems in H\"older spaces that for each fixed $u \in B_1, \lambda \in [0,1]$, the map 
\begin{align*}
    B_1 \to B_2, \quad v \mapsto D_u F(u,\lambda)(v)
\end{align*}
is invertible (with bounded inverse), and so from the implicit function theorem we see that $\Lambda$ is open. We conclude that $\Lambda = [0,1]$, and in particular $1 \in \Lambda$, which completes the proof in the case that $f$, $g$, and $\sigma$ have bounded derivatives of all orders. The general case can now be handled with a mollification procedure and a compactness argument, as in the proofs of Theorem \ref{thm:existence} and \ref{thm:existencedecoup}.
\end{proof}

\begin{proof}[Proof of Theorem \ref{thm.fbsde}]
As explained in Remark \ref{rmk.uniquefbsde}, we need only show existence. It is routine to check that that if $F, H$, $\Sigma$ and $G$ satisfy the assumptions of Theorem \ref{thm.fbsde}, then the data $\sigma, f, g$ given by \eqref{translate} satisfy the conditions of Theorem \ref{thm:existencedecoup}, so we get functions $(u,v = \sigma Du)$ with the following property: with $\tilde{X}$ defined by 
\begin{align*}
    \tilde{X}_t = x_0 + \int_{0}^t \sigma(s,\tilde{X}_s,u(s,\tilde{X}_s)) dB_s, \quad 0 \leq t \leq T, 
\end{align*}
we have
\begin{align} \label{girs}
    u^i(t,\tilde{X}_t) &= g^i(\tilde{X}_T) + \int_t^T \bigg(F^i(s,\tilde{X}_s,u(s,\tilde{X}_s),v(s,\tilde{X}_s)) \nonumber \\ &+v^i(s,\tilde{X}_s) \cdot \sigma^{-1}(s,\tilde{X}_s,u(s,\tilde{X}_s)) H(s,\tilde{X}_s,u(s,\tilde{X}_s),v(s,\tilde{X}_s) \bigg) ds \nonumber \\ &- \int_t^T v^i(s,\tilde{X}_s) \sigma^{-1}(s,\tilde{X}_s, u(s,\tilde{X}_s))  d\tilde{X}_s. 
\end{align}
where $\tilde{B} = B - \int \sigma^{-1}(\cdot, \tilde{X}, u(\cdot, \tilde{X}))dt$ is a Brownian motion under the probability measure $\bQ$ given by $d\bQ = \mathcal{E}(\int \sigma^{-1}(\cdot, \tilde{X}, u(\cdot, \tilde{X})) dB)_T d\bP$. Now define $X$ by 
\begin{align*}
    X_t = x_0 + \int_{0}^t H(s,X_s, u(s,X_s), v(s,X_s)) ds + \int_{t_0}^t \sigma(s,X_s,u(s,X_s)) dB_s, \quad 0 \leq t \leq T. 
\end{align*}
Now by Girsanov there is a probability measure $\bQ$ under which $X$ has the same law as $\tilde{X}$, so the relation 
\begin{align} \label{girs2}
    u^i(t,X_t) &= g^i(X_T) + \int_t^T \bigg(F^i(s,X_s,u(s,X_s),v(s,X_s))  \nonumber  \\ &+v^i(s,X_s) \cdot \sigma^{-1}(s,X_s,u(x,X_s)) H(s,X_s,u(s,X_s),v(s,X_s) \bigg) ds  \nonumber \\ &- \int_t^T v^i(s,X_s) \sigma^{-1}(s,X_s, u(s,X_s))  dX_s.
\end{align} 
holds under $\bQ$, hence also under $\bP$. This is equivalent to 
\begin{align} \label{girs3}
    u^i(t,X_t) &= g^i(X_T) + \int_t^T F^i(s,X_s,u(s,X_s),v(s,X_s)) - \int_t^T v^i(s,X_s) dB_s, 
\end{align} 
i.e. this shows that the triple $(X,Y,Z) = (X, u(\cdot, X), \sigma(\cdot, X) Du(\cdot,X))$ solves \eqref{fbsde}. 
\end{proof}

\bibliographystyle{amsalpha}
\bibliography{fbsde}

\end{document}